\documentclass[11pt, reqno]{amsart}
\usepackage{amssymb, amsmath, amsthm, mathabx, amsfonts, xspace, overpic, color, amsfonts}
\usepackage[dvipsnames]{xcolor}
\usepackage{geometry}
\usepackage[hyperfootnotes=false]{hyperref}
\hypersetup{
  colorlinks,
  citecolor=Purple,
  linkcolor=Purple,
  urlcolor=Purple}
  
\theoremstyle{plain}
\newtheorem{theorem}{Theorem}[section]
\newtheorem{proposition}[theorem]{Proposition}
\newtheorem{lemma}[theorem]{Lemma}
\newtheorem{corollary}[theorem]{Corollary}

\numberwithin{equation}{subsection}
\numberwithin{theorem}{subsection}

\theoremstyle{definition}
\newtheorem{definition}[theorem]{Definition}

\newtheorem*{notation}{\normalfont \textit{{\large Notation}}}

\theoremstyle{remark}

            \newenvironment{claim}[1]{\par\noindent\underline{Claim:}\space#1}{}
      \newenvironment{claimproof}[1]{\par\noindent\underline{Proof:}\space#1}{\hfill $\blacksquare$}
      	  \newcommand{\Poincare}{Poincar\'e\xspace}
	  \newcommand{\set}[1]{\left\lbrace #1\right\rbrace}
	  \newcommand{\sm}{\smallskip}
      \newcommand{\D}{\mathbb{D}}
  
      \makeatletter
      \def\@setcopyright{}
      \def\serieslogo@{}
      \makeatother

\begin{document}

   \author{Bahar Acu}
   \address{Department of Mathematics, University of Southern California, Los Angeles, CA 90089}
   \email{bacu@usc.edu}
   \urladdr{http://www-scf.usc.edu/\char126 bacu/}
   
   \thanks{The first author acknowledges support from the U.S. National Science Foundation. She was partially supported by the NSF grant of Ko Honda (DMS-1406564).}

\author{Russell Avdek}
\email{russell.avdek@gmail.com}

   \title[Symplectic Mapping Class Group Relations]{Symplectic Mapping Class Group Relations Generalizing the Chain Relation}

   \begin{abstract}
   In this paper, we examine mapping class group relations of some symplectic manifolds. For each $n\geq 1$ and $k \geq 1$, we show that the $2n$-dimensional Weinstein domain $W = \{f=\delta\} \cap B^{2n+2}$, determined by the degree $k$ homogeneous polynomial $f\in \mathbb{C}[z_0,\dots,z_n]$, has a Boothby-Wang type boundary and a right-handed fibered Dehn twist along the boundary that is symplectically isotopic to a product of right-handed Dehn twists along Lagrangian spheres. We also present explicit descriptions of the symplectomorphisms in the case $n=2$ recovering the classical chain relation for the torus with two boundary components.
   \end{abstract}

      \maketitle

      \section{Introduction}\label{intro}

About fifteen years ago, Giroux \cite{G} described a correspondence between $(2n+1)$-dimensional contact manifolds and symplectomorphisms of $2n$-dimensional exact symplectic manifolds via a generalization of a construction due to Thurston-Winkelnkemper \cite{TW}. When $n=1$, $\pi_0$ of the symplectomorphism group of a surface is equal to the mapping class group and hence can be studied entirely in terms of Dehn twists along simple closed curves.  Considerable progress has been made in understanding so-called ``Giroux correspondence'' for contact $3$-manifolds; for example, Wand \cite{AW} recently gave a characterization of tightness in terms of the supporting open book decomposition. 
\smallskip

Symplectomorphisms of symplectic manifolds of dimension greater than 2 are comparatively not as well understood, although there has been substantial progress. It had been known that one can construct symplectic Dehn twists ``along'' Lagrangian spheres \cite{A}.  Starting with \cite{P}, Seidel and Khovanov-Seidel systematically studied symplectic Dehn twists in \cite{KS}, \cite{PST}, \cite{P2}, \cite{PS}, and also Wu recently described a classification of Lagrangian spheres in $A_n$-surface singularities in dimension $4$ \cite{WWW}. However, there exist exact symplectic manifolds which do not contain Lagrangian spheres (e.g., the $2n$-dimensional disk). For such manifolds, there is no general way of constructing symplectomorphisms that are not symplectically isotopic to the identity.
\smallskip

Biran and Giroux \cite{BiG} introduced another class of symplectomorphisms, namely the \emph{fibered Dehn twist} which turns out to be a strong tool to study symplectic manifolds that do not contain Lagrangian spheres. More precisely, let $(W, \omega)$ be a symplectic manifold with contact type boundary $(M, \xi=\operatorname{ker}(\lambda))$ carrying a free circle action on $M \times [0, 1]$.  Assume also that the action preserves the contact form $\lambda$. The contact boundary $M$ is then called of \textit{Boothby-Wang type}. A \emph{fibered Dehn twist} along the neighborhood $M \times [0, 1]$ is a boundary-preserving symplectomorphism $\tau$ defined as follows: 
\begin{align*}
\tau: M \times [0,1] &\longrightarrow M \times [0,1], \\
(z,t) &\longmapsto (z\cdot[s(t)\hspace{0.1cm}\mbox{mod}\hspace{0.05cm} 2\pi],t), 
\end{align*}
where $s:[0,1] \rightarrow \mathbb{R}$ is a smooth function such that $s(t)=0$ near $t=1$ and $s(t)=2\pi$ near $t=0$. This will be explained in detail in Section \ref{boothby_wang}. Many such fibered Dehn twists have been shown not to be symplectically isotopic to the identity by a theorem of Biran and Giroux \cite{BiG,CDK}. However in some cases, such as in our main theorem, one can relate fibered Dehn twists with symplectic Dehn twists.
\smallskip

The main goal of this paper is to relate fibered Dehn twists and (regular) symplectic Dehn twists for certain classes of symplectic manifolds. We prove the following theorems which shows that in certain cases, fibered Dehn twists can be expressed as a product of right-handed Dehn twists.\smallskip

Throughout this paper, $B^{2n+2}_{\epsilon}$ denotes a $(2n+2)$-dimensional ball of radius $\epsilon>0$ in $\mathbb{C}^n$ centered at the origin.

   \begin{theorem}\label{main_theorem}
Let $f\in \mathbb{C}[z_0,\dots,z_n]$ be a homogeneous polynomial of degree k with an isolated singularity at 0 for $n \geq 1$, $k\geq1$. Denote by $(W,d\beta)$ the $2n$-dimensional Weinstein domain, where for $\epsilon \geq 0$ small and $\delta=\delta(\epsilon)>0$ small
\begin{center}
$W = \{f(z_0,\dots,z_n)=\delta\} \cap B^{2n+2}_{\epsilon} \hspace{.6cm} \mbox{and} \hspace{.6cm} \beta = \dfrac{1}{2} \displaystyle\sum_{j=0}^{n} (x_jdy_j-y_jdx_j)$.
\end{center}
Then $W$ has Boothby-Wang type boundary and $\partial W$ has a coherent open book decomposition $OB(F, \Phi_{\partial})$ such that a right-handed fibered Dehn twist $\Phi_{\partial} \in \mbox{Symp}(F, d\beta, \partial F)$ is boundary-relative symplectically isotopic to a product of $k(k-1)^{n}$ right-handed Dehn twists $\Phi_{1}\circ \dots \circ \Phi_{k(k-1)^{n}}$ along Lagrangian spheres. Here $F=W \cap \{z_n=\nu\}$ is a degree $k$ hypersurface in $W$ for $\nu >0$ small.
   \end{theorem}

It turns out that when we restrict ourselves to a certain class of homogeneous polynomials of degree $k$, we obtain a beautiful relation between the fibered Dehn twist $\Phi_{\partial}$ and symplectic Dehn twist $\Phi$.  We obtain the following immediate corollary of Theorem \ref{main_theorem}.
  
 \begin{corollary}[Roots of fibered Dehn twists]\label{main_corollary} 
  Let $f(z_0,\dots, z_n)=z_0^k+ \dots+ z_n^k \in  \mathbb{C}[z_0,\dots,z_n]$ for $n \geq 1$, $k\geq1$. Denote by $(W,d\beta)$ the $2n$-dimensional Weinstein domain, where for $\epsilon \geq 0$ small and $\delta=\delta(\epsilon)>0$ small
\begin{center}
$W= \{f(z)=\displaystyle\sum_{j=0}^{n}z_{j}^{k} = \delta \}\cap B^{2n+2}_{\epsilon} \hspace{.6cm} \mbox{and} \hspace{.6cm} \beta = \frac{1}{2} \displaystyle\sum_{j=0}^{n} (x_jdy_j-y_jdx_j)$. 
\end{center}
Then $W$ has Boothby-Wang type boundary and $\partial W$ has a coherent open book decomposition $OB(F, \Phi_{\partial})$ such that a right-handed fibered Dehn twist $\Phi_{\partial} \in \mbox{Symp}(F, d\beta, \partial F)$ is boundary-relative symplectically isotopic to $\Phi^{k}$, where $\Phi$ is a product of $(k-1)^n$ right-handed Dehn twists along Lagrangian spheres. That is, $\Phi$ is the $k$-th root of $\Phi_{\partial}$.
   \end{corollary}

In Section \ref{Corollary3.1}, we prove the following corollary in the case $n=2$, $k=3$; this recovers the classical chain relation for a torus with two boundary components by using the star relation \cite{SG} which will be explained in Section \ref{MCG}, the last section of the paper. Similar representations have been studied in terms of Artin groups by Matsumoto \cite{Mat}.

\begin{corollary}\label{mapping_class_relationn}
Consider the genus 1 surface with 3 boundary components $S_{1,3}$ equipped with the embedded curves $\alpha_{b}, \alpha_{g}, \alpha_{p}, \alpha_{r},b_{1},b_{2}$, and $b_{3}$ as depicted in Figure \ref{mapping_class_relation}. Let  $D_{\alpha_i}$ and $D_{b_j}$ represent Dehn twists along the curves $\alpha_i$ and $b_j$, respectively, where $i \in \{b, g, p, r\}$ and $j \in \{1, 2, 3\}$. Then
\begin{equation*}
(D_{\alpha_{r}}\circ D_{\alpha_{p}} \circ D_{\alpha_{b}} \circ D_{\alpha_{g}} )^{3} = D_{b_{1}}\circ D_{b_{2}} \circ D_{b_{3}}
\end{equation*}
in the mapping class group of $S_{1,3}$. 
\end{corollary}

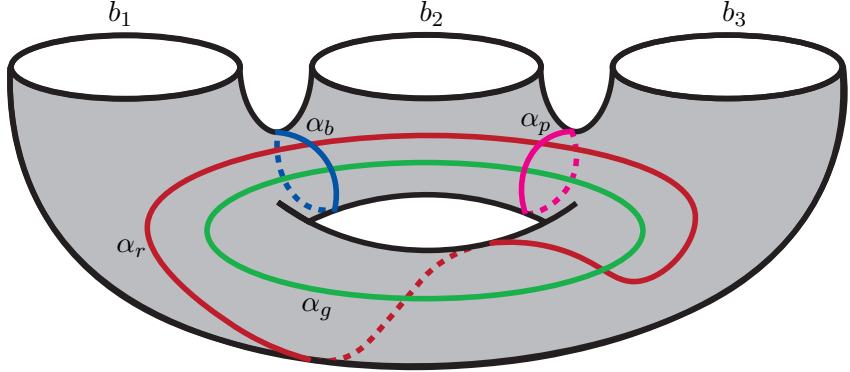
\begin{figure}[h]
\vspace{3mm}
\begin{overpic}[scale=.7,tics=10]{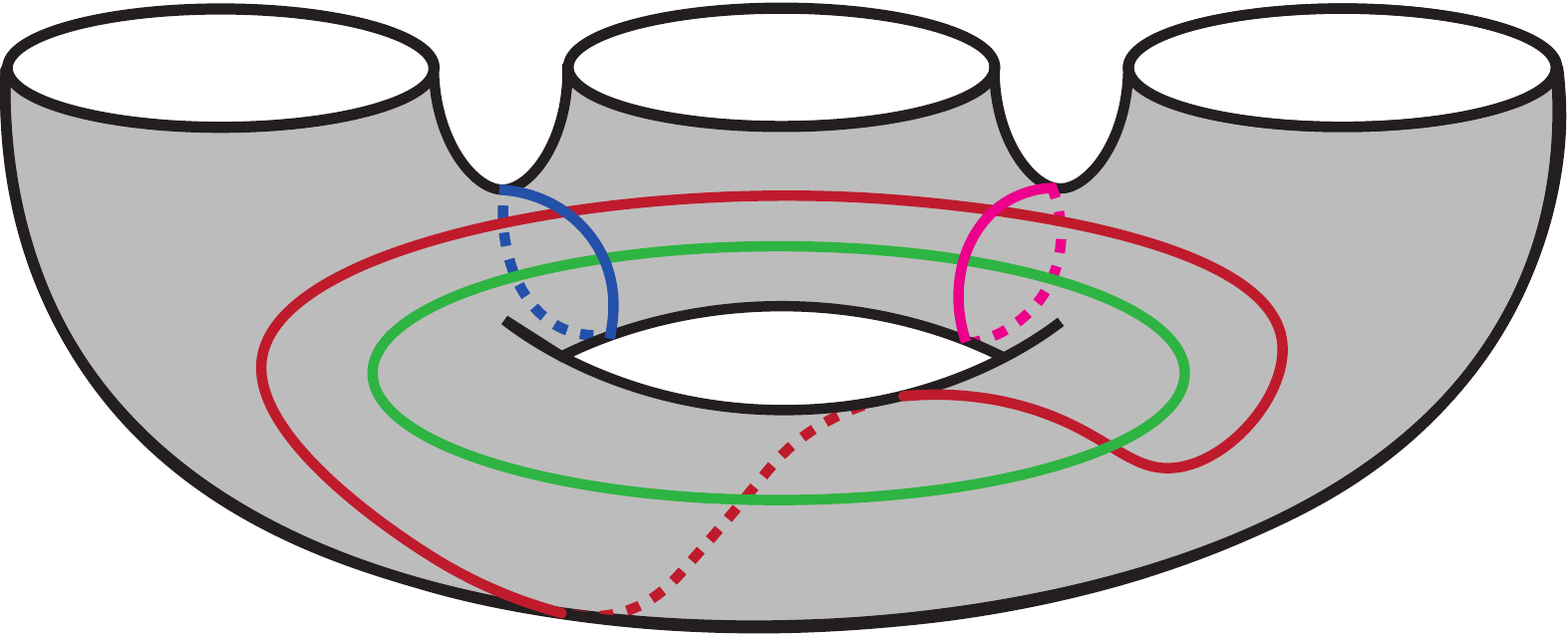}
\put(35.5,28.5){$\alpha_b$}
\put(61,29){$\alpha_p$}
\put(35,7){$\alpha_g$}
\put(13,14){$\alpha_r$}
\put(12,41.5){$b_1$}
\put(49,41.5){$b_2$}
\put(85,41.5){$b_3$}
\end{overpic}
	\caption{We label the blue, green, purple, and red curves as $\alpha_{b}$, $\alpha_{g}$, $\alpha_{p}$, and $\alpha_{r}$, respectively.  The boundary components are labeled as $b_{1}$, $b_{2}$, and $b_{3}$.}
    \label{mapping_class_relation}
\end{figure}

Finally, we will examine contact manifolds given by a weighted homogeneous polynomial. Let $f\in \mathbb{C}[z_1,\dots,z_n]$ be a weighted homogeneous polynomial with an isolated singularity at 0, $n\geq 1$,  and $L(f)=\{f(z_1, \dots, z_n)=0\} \cap S^{2n-1}$ be the corresponding link of singularity. Then the Milnor fibration theorem \cite[Theorem 4.8]{M} states that 
\begin{align*}
\phi: S^{2n-1} - L(f) &\rightarrow S^1, \\
z &\mapsto \dfrac{f(z)}{|f(z)|}
\end{align*}
is a fibration, called the \textit{Milnor fibration} of $f$. The closure of each fiber is a compact manifold with boundary called the \textit{Milnor fiber}. \sm

The Milnor fibration of $f$ then defines an open book decomposition of $S^{2n-1}$, with binding $L(f)$ and fibration $\phi$, called the \textit{Milnor open book} for $S^{2n-1}$. Moreover, this open book supports the standard contact structure $\xi_0=\operatorname{ker}\beta$ on $S^{2n-1}$. To see this, one needs to show that $\beta$ is a contact form on the binding $L(f)$ and $d\beta$ is positively symplectic on every page. Recall that the binding, link of singularity of $f$, is contactomorphic to the standard sphere. Hence, it is left to show that $d\beta$ is nondegenerate on every page. Note that the flow of the Reeb vector field $R$ for $\beta$ is given by $$e^{it}\cdot (z_1, \dots, z_n)=(e^{it/a_1}z_1, \dots, e^{it/a_n}z_n)$$ which is also a diffeomorphism from one page to another. Then the Reeb vector field $R$ for $\beta$ can be computed as follows: 
$$R=\frac{d}{dt}(e^{it/a_1}z_1, \dots, e^{it/a_n}z_n)=\sum_{i=1}^n \frac{1}{a_i}(x_i\frac{\partial}{\partial y_i}-y_i\frac{\partial}{\partial x_i}).$$
Observe that $d\beta(R, \cdot)=0$ and $\beta(R)=1$. The Reeb vector field is transverse to the pages, hence $d\beta$ is nondegenerate on pages.\sm

Unlike Theorem \ref{main_theorem} which examines open books of the boundary restrictions of Milnor fibers, we will consider Milnor open books of the standard contact sphere and prove the following theorem.
   
   \begin{theorem}[Fractional twists] \label{weightedthm}
The Milnor open book for the standard contact sphere $(S_{\epsilon}^{2n-1}, \xi_0=ker\beta)$ is the contact open book
\begin{center} 
$OB(\tilde{F}, \Phi_{\partial})$
\end{center}
where the page $\tilde{F}=\{f(z)=\delta\} \cap B_{\epsilon}^{2n}$ is a Weinstein domain and the monodromy $\Phi_\partial \in  \mbox{Symp}(\tilde{F}, d\beta, \partial \tilde{F})$ is a generalization of a fractional fibered Dehn twist along the boundary that is symplectically isotopic to a product of $\mu(f,0)$ right-handed Dehn twists along Lagrangian spheres.
 \end{theorem}
 
Here $\tilde{F}=\{f=\delta\}\cap B_{\epsilon}^{2n}$, namely the Milnor fiber, is a smoothing of the (singular) Weinstein domain $F=\{f=0\} \cap B_{\epsilon}^{2n}$, and $\mu(f,0)$ is the Milnor number of $f$ at the isolated singularity at the origin. See Section \ref{proof_of_main_theorem} for more details.\\

Although this work was done completely independently, after the paper appeared on the arXiv, it was brought to our attention that there was substantial overlap with Seidel \cite{PS}, in which he more or less proved Theorem \ref{weightedthm}.\\

\noindent\textbf{Plan of the Paper}

\begin{itemize}

\item In Section \ref{preliminaries}, we give basic definitions.
\item In Section \ref{boothby_wang}, we describe the Boothby-Wang circle bundle and a right-handed fibered Dehn twist explicitly and present the conditions for a Boothby-Wang circle bundle to possess a supporting open book and a monodromy as a right-handed fibered Dehn twist. 
\item In Section \ref{proof_of_main_theorem}, we give the proof of the main theorem by discussing the contact open book coming from the Lefschetz fibration with a product of right-handed Dehn twists as monodromy and the Boothby-Wang circle bundle with a right-handed fibered Dehn twist as monodromy. We also give the proof of Corollary \ref{main_corollary}, Theorem \ref{weightedthm}, and the exact number of right-handed Dehn twists appeared in both Theorem \ref{main_theorem} and Corollary \ref{main_corollary}.
\item In Section \ref{n_2}, we discuss the explicit case $n=2$ describing some mapping class group relations for surfaces and provide a proof of Corollary \ref{mapping_class_relationn}. \\
\end{itemize}

\noindent\textbf{Acknowledgements.} We would like to thank our advisor Ko Honda whose support and guidance were central throughout the completion of this project. We also thank Otto van Koert for his invaluable comments and edits on the first version of this paper and also for the instructive discussions on Section \ref{number}, Ailsa Keating for bringing Seidel's work to our attention, and Dan Margalit for his remark regarding mapping class group relations described in Section \ref{MCG}. We thank the referee for careful reading and several suggestions that have improved the exposition of our paper.

\section{Preliminaries}\label{preliminaries}

\begin{definition}A polynomial $f(z_1, \dots, z_n)=\Sigma{c_I z^I}$, where $I=(a_1,\dots,a_n)$ and $z=z_1^{a_1} \dots z_n^{a_n}$, is called a \emph{weighted homogeneous polynomial} or \emph{quasi-homogeneous polynomial} if there exists $n$ integers $w_1,\dots, w_n$ such that $d=w_{1}a_1+\dots+w_{n}a_n$ holds for each monomial of $f$.
\end{definition}
Integers $w_1, \dots, w_n$ are called \emph{weights} of the variables and the sum $d$ is called the \emph{degree} of the polynomial $f$. Observe that $f$ is a homogeneous polynomial when $w_1=\dots=w_n=1$.\sm

Equivalently, $f$ is a \emph{weighted homogeneous polynomial} if and only if 
\begin{equation*}
f(\lambda^{w_1}z_1, \dots, \lambda^{w_n}z_n) = \lambda^{d}f(z_1, \dots, z_n).
\end{equation*}

One can then discuss a locally free circle action on the variety $\{f=0\}$. Therefore, we obtain  a $\mathbb{C}^*$-action on $\mathbb{C}^n - \{0\}$ by 
\begin{align*}
\mathbb{C}^* \times \{f=0\} &\longrightarrow \{f=0\},\\
\lambda \cdot  (z_1, \dots, z_n) & \longmapsto (\lambda^{w_1}z_1, \dots, \lambda^{w_n}z_n).
\end{align*}
\noindent for each $\lambda \in S^1$.

\begin{definition} Let $f \in \mathbb{C}[z_0,\dots,z_n ]$ be a hypersurface singularity germ at the origin in $\mathbb{C}^n$, where $n \geq 0$. Let $S_{\epsilon}^{2n+1}=\partial B_{\epsilon}^{2n+2}$ be the boundary of the closed ball with radius $\epsilon$ centered at the origin and let $L(f)=\{f=0\} \cap S_{\epsilon}^{2n+1}$ be the corresponding link of singularity. Then the smooth locally trivial fibration
\begin{align*}
\phi:  S_{\epsilon}^{2n+1}- L(f) &\longrightarrow S^{1},\\
z &\longmapsto \ \frac{f(z)}{|f(z)|}.
\end{align*}
\noindent is called the \emph{Milnor fibration} of the function germ $f$. This is due to \cite{M}.
\end{definition}

When $f$ has an isolated singularity at the origin, then the link $L(f)$ is a smooth manifold. Moreover, any fiber $F_t=\phi^{-1}(t)$ is a smooth open manifold. In particular, the closure of each fiber, namely the \emph{Milnor fiber}, is a compact manifold with boundary $L(f)$ \cite{M}.\sm

Note that if $f$ has an isolated singularity at the origin, then the associated Milnor fiber $F$ has the homotopy type of a bouquet of $n$-spheres \cite{M}. The number of spheres in this bouquet is called the \emph{Milnor number}, denoted by $\mu(f, 0)$. \smallskip

 If $f$ is a weighted homogeneous polynomial with an isolated singularity at the origin and with weights $(w_1, \dots, w_n, d)$, then the Milnor number can be computed in terms of the weights \cite{AD} as follows:
\begin{equation*}
\mu(f, 0)=\prod_{i=1}^{n} \dfrac{d-w_i}{w_i}.
\end{equation*}

\begin{definition}Let $W$ be a $2n$-dimensional compact manifold with boundary. A {\em{Weinstein domain structure}} on $W$ is a triple $(\omega, Z, \phi)$, where 
\begin{enumerate}
\item $\omega=d\beta$ is an exact symplectic form on $W$,
\item $Z$ is a Liouville vector field for $\omega$ defined by $\beta=\iota_Z\omega$, 
\item $\phi : W \rightarrow \mathbb{R}$ is a generalized Morse function for which $Z$ is gradient-like,
\item $\phi$ has $\partial W=\{\phi=0\}$ as a regular level set.
\end{enumerate}
	The quadruple $(W, \omega, Z, \phi)$ is then called a \emph{Weinstein domain}.
\end{definition}

\begin{definition} A \emph{Lefschetz fibration} is a smooth map $g: X \rightarrow \D $, where $X$ is a $2n$-dimensional compact manifold with boundary and $\D$ is a 2-disk, with the following properties:
\begin{enumerate}
\item The critical points of $g$ are isolated, nondegenerate, and are in the interior of $X$.
\item If $p\in X$ is a critical point of $g$, then there are local complex coordinates $(z_{1},\dots , z_{n})$ about $p=(0,\dots, 0)$ on $X$ and $z$ about $g(p)$ on $\D$ such that, with respect to these coordinates, $g$ is given by the complex map $z = g(z_1, \dots, z_n) = z_{1}^{2} +\dots +z_{n}^{2}$.
\end{enumerate}
\end{definition}
The preimage of a regular value of $g$ is a $(2n-2)$-dimensional submanifold of $X$.

\section{Open Books for Boothby-Wang Bundles} \label{boothby_wang}

\begin{definition} An \textit{abstract open book decomposition} is a pair $(F, \Phi)$, where
\begin{enumerate}
\item $F$ is a compact $2n$-dimensional manifold with boundary, called the \textit{page} and
\item $\Phi: F \rightarrow F$ is a diffeomorphism preserving $\partial F$, called the \emph{monodromy}.
\end{enumerate} 
\end{definition}

\begin{definition}An \emph{open book decomposition} of a compact oriented manifold $Y$ is a pair $(B,\pi)$, where
\begin{enumerate}
\item $B$ is a codimension 2 submanifold of $Y$ with trivial normal bundle, called the \emph{binding} of the open book.
\item $\pi: Y - B \rightarrow S^{1}$ is a fiber bundle of the complement of $B$ and  the fiber bundle $\pi$ restricted to a neighborhood $B \times \mathbb{D}$ of $B$ agrees with the angular coordinate $\theta$ on the normal disk $\mathbb{D}$. 
\end{enumerate}
\end{definition}

Define $F_{\theta}:= \overline{\pi^{-1}(\theta)}$ and observe that $\partial F_{\theta}=B$ for all $\theta \in S^1$. We call the fiber $F=F_\theta$, for any $\theta$, a \emph{page} of the open book. The holonomy of the fiber bundle $\pi$ determines a conjugacy class in the orientation preserving diffeomorphism group of a page $F$ fixing its boundary, i.e., in $\mbox{Diff}^{+}(F,\partial F)$ which we call the {\em{monodromy}}. \sm

By using this description, one can construct a closed oriented $(2n+1)$-dimensional manifold $M$ from an abstract open book in the following way: Consider the mapping torus $F_{\Phi}=[0,1] \times F / (0, \Phi(z)) \sim (1, z)$. We set
\begin{equation*}
M_{(F, \Phi)}= F_{\Phi}\cup_{\partial F_{\Phi}} \left( \bigsqcup_{|\partial F_{\Phi}|} \partial F \times \mathbb{D} \right), \\
\end{equation*}
by gluing $\partial (\partial F \times \mathbb{D}) = \partial F \times S^1$ to each boundary component of $F_{\Phi}$, where $|\partial F_{\Phi}|$ is the number of boundary components of $F_{\Phi}$. Here  the boundary of each disk $\{pt\} \times S^1$ in $\partial F \times \mathbb{D}$ gets glued to $ S^1 \times \{pt\}$ in the mapping torus. Then $(F, \Phi)$ is an open book decomposition of a closed oriented $(2n+1)$-dimensional manifold $M$ if $M_{(F, \Phi)}$ is diffeomorphic to $M$. \smallskip

Note that the mapping torus $F_{\Phi}$ carries the structure of a smooth fibration $F_{\Phi} \rightarrow S^1$, away from $\partial F$, whose fiber is the interior of the page $F$ of the open book decomposition. 

\begin{definition} The \emph{symplectomorphism group} of $(F, \omega)$ consists of all orientation preserving diffeomorphisms $ \phi: F \stackrel\sim\rightarrow F$ such that $\phi$ preserves the symplectic form $\omega$, i.e., $\phi ^{\star}(\omega)=\omega$ and $\phi|_{\partial F}=id$. It is denoted by $\mbox{Symp}(F,\omega,\partial F) \subset \mbox{Diff}^{+}(F,\partial F)$.
\end{definition}

\begin{definition} A contact structure $\xi$ on a manifold $Y$ is said to be \emph{supported by an open book} $(B,\pi)$ of $Y$ if it is the kernel of a contact form $\lambda$ satisfying the following:

\begin{enumerate}
\item $\lambda$ is a positive contact form on the binding and
\item $d\lambda$ is positively symplectic on every page.
\end{enumerate}
If these two conditions hold, then the open book $(B,\pi)$ is called a \emph{supporting open book} for the contact manifold $(Y, \xi)$ and the contact form $\lambda$ is said to be \emph{adapted} to the open book $(B,\pi)$ or $\lambda$ is \emph{adapted by} the fiber bundle $\pi$.
\end{definition}

As mentioned in Section \ref{intro}, Giroux \cite{G} and Thurston-Winkelnkemper \cite{TW} showed that every open book decomposition gives rise to a supporting contact manifold. More precisely, to each triple $(F^{2n}, \lambda, \Phi)$, where $(F,d\lambda)$ is a Weinstein domain and $\Phi\in \mbox{Symp}(F, d\lambda, \partial F)$, we can associate a contact manifold $(Y^{2n+1},\xi)$.

\subsection{Boothby-Wang Circle Bundles or Prequantization Circle Bundles} Let $M^{2n}$ be a compact integral symplectic manifold with symplectic form $\omega$, where $[\omega] \in H^{2}(M;\mathbb{Z})$. Then there exists a complex line bundle $C$ over $M$ with $c_1(C)=[\omega]$ as the first Chern class classifies complex line bundles.\sm

 Consider the  associated principal circle bundle of $C$ given by $\Pi: P \rightarrow M$. Then according to \cite[Theorem 3]{BW}, there exists a connection $1$-form $\beta$ on $P$ such that 
\begin{enumerate}
\item $-2\pi \Pi^{*}\omega = d\beta$ and 
\item$\beta$ is a contact $1$-form on P.
\end{enumerate}

The contact form $\beta$ is called a  {\em{Boothby-Wang form}} carried by the circle bundle $P \xrightarrow {\Pi} M‎$. The fact that the contact form $\beta$ is a connection $1$-form implies that the vector field, $R_{\beta}$, generating the circle action satisfies $\mathcal{L}_{R_{\beta}}\beta \equiv 0$ and $\beta(R_{\beta}) \equiv 1$. This implies that $R_{\beta}$ is the Reeb vector field for $\beta$.
\vspace{1mm}
The circle bundle $(P \xrightarrow {\Pi} M‎, \beta)$ is then called a \emph{Boothby-Wang circle bundle} associated with $(M,\omega)$. Since the curvature $2$-form $\omega$ on $M$ makes the base space $M$ into a symplectic manifold, $(P, \beta)$ is called prequantum circle bundle (or Boothby-Wang circle bundle) that provides a geometric prequantization of $(M,\omega)$. This is due to Boothby and Wang, 1958 \cite{BW}.

\subsection{Fibered Dehn Twist} Let $F$ be a symplectic manifold with contact type boundary $\partial{F}$ carrying a free $S^{1}$-action on the neighborhood $\partial{F} \times [0,1]$ that preserves the contact form $\beta$ on $\partial{F}$. Here $\partial F\times [0,1]$ is a collar neighborhood of $\partial F=\partial F\times\{1\}$. Then we can define a \emph{right-handed fibered Dehn twist} along the boundary $\partial{F}$ as follows: 
\begin{align*}
\tau: \partial{F} \times [0,1] &\longrightarrow \partial{F} \times [0,1], \\
(z,t) &\longmapsto (z\cdot[s(t)\hspace{0.1cm}\mbox{mod}\hspace{0.05cm} 2\pi],t), 
\end{align*}

\noindent where $s:[0,1] \rightarrow \mathbb{R}$ is a smooth function such that $s(t)=0$ near $t=1$ and $s(t)=2\pi$ near $t=0$. It is easily verified that this is a symplectomorphism of $(\partial{F} \times [0,1],d(e^{t}\beta))$ which is equal to identity near boundary whose isotopy class is independent of the choice of $s$. See \cite{CDK} for more details.\sm

This notion was introduced by Biran and Giroux to emphasize that there are examples of symplectomorphisms such as fibered Dehn twists, which are not necessarily products of Dehn twists; see \cite{BiG}.

\subsection{Fractional Fibered Dehn Twist} Here we refer the reader to \cite{CDK2} for a more comprehensive discussion. The setup in this subsection is as follows: We will denote by $k$ the degree of the weighted homogeneous polynomial $f$ with an isolated singularity at 0 and by $F=\{f(z)=\delta\}\cap B^{2n}$ the associated Milnor fiber. Denote by $P$ the boundary of $F$ and let $l$ be any positive integer that divides $k$ and $(\tilde{P} \xrightarrow {\Pi} M‎, \beta)$ be a Boothby-Wang bundle over $(M, \frac{k}{l}\omega)$, where $\tilde{P}$ is a covering of $P$ given by the following $l$-fold covering map:
\begin{align*}
pr_P: \tilde{P} &\longrightarrow P, \\
p &\longmapsto p \otimes \dots \otimes p
\end{align*} 

Note that the map $pr_P$ is well-defined since $H^2(M; \mathbb{Z})$ is torsion-free as $l$ divides $k$.  Next, we want to extend the covering from $P=\partial F$ to the symplectic manifold $F$.\sm

\begin{definition}
A covering map $pr_F: \tilde{F} \mapsto F$ is called \emph{adapted} to $pr_{P}: \tilde{P} \mapsto P$ if $pr_F$ restricts to $pr_{P}$ on a collar neighborhood $P$.
\end{definition}

One can always find such an adapted $l$-fold covering $pr_F: \tilde{F}\rightarrow F$ for a $2n$-dimensional Weinstein manifold $F$ with $2n\geq 6$, because there exists a subgroup of index $l$ in $\pi_1(F)=\pi_1(\partial F)$ (this equality is true for dimensions 6 and higher) as $l | k$.\sm

Given an adapted covering, one can lift a fibered Dehn twist $\tau$ to a $\tilde{\tau}$, where
\begin{align*}
\tilde{\tau}: \partial \tilde{F} \times I &\longmapsto \partial \tilde{F} \times I, \\
(z, t) & \longmapsto (z \cdot[s(t)\hspace{0.1cm}\mbox{mod}\hspace{0.05cm} 2\pi],t)
\end{align*} 

\noindent where $s:[0,1] \rightarrow \mathbb{R}$ is a smooth function such that $s(t)=0$ near $t=1$ and $s(t)=\frac{2\pi}{l}$ near $t=0$ for a positive integer $l$. This implies that there is  an  action  of  $\mathbb{Z}_l$ given by $s(t)$ which induces rotation by roots of unity in some collar neighborhood $\partial \tilde{F} \times I$ of $\partial \tilde{F}$. Hence, the map $\tilde{\tau}$ can similarly be extended to the interior of $F$ by using deck transformation of the cover $\tilde{F} \rightarrow F$  which generates this $\mathbb{Z}_l$ action in the interior of the page. Observe that this action also rotates the margin of the page by $\frac{2\pi}{l}$.\sm

The map $\tilde{\tau}$ is symplectomorphism, see \cite{CDK2}. Observe also that for $l=1$, the map $\tilde{\tau}$ is a fibered Dehn twist.

\begin{definition}
The symplectomorphism $\tilde{\tau}$ given as above is called a \emph{right-handed fractional fibered Dehn twist} of power $l$ or a \emph{fractional twist} as defined in \cite{CDK2}.
\end{definition}

\subsection{Setup} Now we will provide a description of an adapted open book decomposition of a Boothby-Wang circle bundle over a symplectic manifold. This construction has been introduced by Biran and Giroux \cite{BiG} and generalized and clarified by Chiang, Ding and van Koert \cite{CDK}.\sm

Any integral symplectic manifold $(M^{2n}, \omega)$ has a compact symplectic hypersurface $H$ whose homology class $[H] \in H_{2n-2}(M;\mathbb{Z})$ is \Poincare dual to the symplectic class $[k\omega]\in H^{2}(M;\mathbb{Z})$ such that $F=M- \nu(H)$, the complement of an open tubular neighborhood $\nu(H)$ of $H$, carries a Weinstein domain structure, \cite{D} and \cite{G}. If one can patch $F$, $\partial F \times [0, 1]$, and $\partial F \times \mathring{\mathbb{D}}^2$ together, then the hypersurface $H$ is called an \emph{adapted Donaldson hypersurface}.  See \cite[Section 6.1]{CDK} for the patching argument. \sm

Now consider the Weinstein domain $(W,d\beta)$ with a Boothby-Wang type boundary $(P, \beta|_P)$. Then there are two constructions of a contact manifold using the given data. Firstly, we can construct a Boothby-Wang circle bundle $(P \xrightarrow {\Pi} M‎, \beta)$ over an integral symplectic manifold $M$. We can also define a right-handed fibered Dehn twist $\Phi_{\partial}$ along the boundary of $F$, then define a contact open book of $\partial W$ with page $F=M-H$ and monodromy as $\Phi_{\partial}$. Thus, we can talk about the relation between the Boothby-Wang circle bundle and its associated open book. That relation is given by the following theorem.

\begin{theorem} \label{theorem 3.1}\cite[Theorem $6.3$ and Corollary $6.4$]{CDK} Let $(W,d\beta)$ be a Weinstein domain with a Boothby-Wang type boundary $(\partial W=P, \beta|_{P})$ and $\omega$ be an integral symplectic form on $M$. Let $H$ be an adapted Donaldson hypersurface \Poincare dual to $[\omega]$, $F=M-H$, and $\Phi_{\partial}$ be a right-handed fibered Dehn twist along $\partial F$. Then
\begin{enumerate}
\item The Boothby-Wang circle bundle  $(P \xrightarrow {\Pi} M‎, \beta)$ associated with $(M, \omega)$ has an open book decomposition whose monodromy is a right-handed fibered Dehn twist.
\item The open book $(F, \Phi_{\partial})$ with page $F$ and monodromy $\Phi_{\partial}$ is contactomorphic to the Boothby-Wang circle bundle $(P \xrightarrow {\Pi} M‎, \beta)$ associated with $(M, \omega)$. 
\end{enumerate}
\end{theorem}

\begin{definition} An open book decomposition $(B,\pi)$ is \emph{coherent} with a Boothby-Wang circle bundle $(P \xrightarrow {\Pi} M‎, \beta)$ if 
\begin{enumerate}
\item The contact form $\beta$ is adapted by the fibration $\pi: P -\Pi^{-1}(H)\rightarrow S^{1}$, where $\Pi: P \rightarrow M$ is the Boothby-Wang circle bundle.
\item The binding $B=\Pi^{-1}(H)$ for some adapted Donaldson hypersurface $H \subset M$.
\item The monodromy $\Phi_{\partial}$ is a right-handed fibered Dehn twist.
\end{enumerate}
\end{definition}
   
   \section{Proofs of Theorem \ref{main_theorem} and \ref{weightedthm}}\label{proof_of_main_theorem} 
   \label{main}
   
   Throughout, we denote the coordinates on $\mathbb{C}^{n+1}$ by $z_j = x_j + iy_ j $, $j = 0,\dots,n$, and the open book with page $F$ and monodromy $\Phi$ by $OB(F, \Phi)$.\sm
   
   The proof is by comparing two points of view, namely the Lefschetz fibration point of view and the Boothby-Wang fibration point of view. Throughout, $f$ is a homogeneous polynomial of degree $k$ with an isolated singularity at 0 and $W=\{f(z_0, \dots, z_n )=\delta\} \cap B^{2n+2}_{\epsilon}$ is the Weinstein domain unless otherwise is stated.
   
   \subsection{Lefschetz Fibration Construction} Denote by $z_n : W \rightarrow \mathbb{C}$ the projection map to the last coordinate with an isolated singularity. Then the restriction map, denoted by $h:= z_{n}|_{\{|z_n|\leq \epsilon\}} : W\cap \{|z_n|\leq \epsilon\} \to \{|z_n|\leq \epsilon\}$, for $\epsilon$ small, is a Lefschetz fibration after Morsifying its critical points which will be explained in the next paragraph and  $F = W \cap \{z_{n}= \epsilon\}$ is the generic fiber of the Morsification of this fibration. \sm
   
   Since all critical points of a Lefschetz fibration are of Morse-type, we need to perturb $z_n$ in a certain stable way so that the isolated degenerate singularity at 0 will split up into other isolated singularities which are non-degenerate. This deformation is called \textit{Morsification} which is a small deformation of $z_n$ having $r$ distinct Morse critical points $p_1, \dots , p_r$ and critical values $c_1, \dots , c_r$. Here $z_n + \eta g$ is a Morsification of $z_n$ if $g(z)$ is a linear function in general position and $\eta$ is very small. We also know that there are finitely many such Morse-type critical points, i.e., $r$ is finite, after Morsification of $z_n$ since such critical points are isolated. One can also prove that the exact number of such critical points is $r=k(k-1)^n$. For the proof, we refer the reader to Section \ref{number}.\smallskip

That is to say, we have finitely many Lagrangian spheres in the fiber coming from critical values. Moreover, it gives rise to associated monodromies. Denote the monodromy associated to critical values by $\Phi_{1}\circ \dots \circ \Phi_{k(k-1)^n} $, where each $\Phi_{i}\in \mbox{Diff}^{+}(F,\partial F),  i=1,\dots,k(k-1)^n$, is a right-handed Dehn twist. Then we have the following open book decomposition associated to the boundary restriction of the Lefschetz fibration $h$:
   
\begin{itemize}
\item Page: $F = W \cap \{z_{n}=\epsilon \}$.
\item Binding: $B = \partial W \cap \{z_{n}=0\}$.
\item Monodromy: $\Phi_{1}\circ \dots \circ \Phi_{k(k-1)^n}$.
\end{itemize}

Here observe that $F$ is a retracted page of the open book decomposition\\ $(B, \Phi_{1}\circ \dots \circ \Phi_{k(k-1)^n})$. By abuse of notation, we refer to a retracted page as a page of the open book. \sm

Let us consider the boundary of the Weinstein domain $W\cap \{|z_{n}| \leq \epsilon\}$. We know that an open book naturally arises as boundary restriction of a Lefschetz fibration. According to the Lefschetz fibration construction above, $\partial (W \cap \{|z_{n}| \leq \epsilon\}) $  is supported by the open book decomposition with page $F$ and monodromy  $\Phi_{1}\circ \dots \circ\Phi_{k(k-1)^n}$. The binding of this open book is the intersection of the link associated to the homogeneous polynomial $f$, i.e., $\partial W = \{f=0\} \cap S^{2n+1}$ with the hyperplane $\{z_n =0\}$. \sm

The analysis above provides a description of the monodromy coming from the Lefschetz fibration point of view. In order to analyze the desired isotopy between the two monodromies mentioned in the main theorem, it suffices to describe the monodromy coming from the Boothby-Wang point of view, namely a right-handed fibered Dehn twist. In order to do that, we need the data given in Theorem \ref{theorem 3.1}.

\subsection{Boothby-Wang Construction} Here we closely follow the description in \cite{CDK}. Consider the Weinstein domain $W$ with the boundary $\partial W$ carrying a free $S^1$-action which preserves the contact form. We can extend this action to the neighborhood $\partial W \times [0,1]$ so that the contact form $\beta$ on $\partial W$ is preserved. \smallskip

Let $\Pi: \partial W \rightarrow M$ be a principal circle bundle such that the $S^1$-action preserves the contact form. Then circle fibers are closed orbits of the Reeb vector field $R_{\beta}$, i.e., Reeb foliates circle fibers. So we can view the boundary of the Weinstein domain $W$ lying in $S^{2n+1}$ as a Boothby-Wang type boundary. \sm

Now consider the fibration 
\begin{align*}
\pi:\partial W - \{z_{n}=0\} &\to S^{1}, \\
 (z_{0},\dots,z_{n}) &\mapsto \frac{z_{n}}{|z_{n}|} .
\end{align*}

Observe that $H=\{z_{n}=0\}$ is an adapted Donaldson hypersurface in $M$ for which  $\Pi^{-1}(\{z_{n}=0\})=\partial W \cap \{z_{n}=0\}$ is the binding of the open book decomposition of $\partial W$ with page $F=\pi^{-1}(\theta_{0})$ for some fixed angle $\theta_{0}=arg(z_{n})$ and monodromy a right-handed fibered Dehn twist. The Boothby-Wang circle bundle $(\partial W \xrightarrow {\Pi} M‎, \beta)$ is then contactomorphic to a contact structure supported by the open book decomposition of $\partial W$ by Theorem \ref{theorem 3.1}.\sm
 
 Also, the open book decomposition discussed above is coherent with the Boothby-Wang circle bundle $(\partial W \xrightarrow {\Pi} M‎, \beta)$ since the contact form on $\partial W$ is adapted to the open book $(B=\Pi^{-1}(\{z_{n}=0\}), \pi)$ and the monodromy is a right-handed fibered Dehn twist. \sm

  So, the coherent open book decomposition is as follows:
\begin{itemize}
\item Page: $F=(\partial W-\{z_{n}=0\})  \cap \{arg(z_{n})=\theta_{0}\}$ for some fixed $\theta_{0}$.
\item Binding: $B= \partial W \cap \{z_{n}=0\}$.
\item Monodromy: $\Phi_{\partial}$, a right-handed fibered Dehn twist.
\end{itemize}

\subsection{Comparison of Two Views} Observe that the pages and the bindings of the two open books discussed above are exactly the same up to symplectomorphism. To finish the proof, we have to check that the monodromy coming from the Lefschetz fibration point of view (product of Dehn twists along Lagrangian spheres) is symplectically isotopic to the monodromy coming from the Boothby-Wang fibration point of view (right-handed fibered Dehn twist) relative to the boundary. In order to show that, we will construct a symplectic isotopy so that we have the desired equivalence.\sm

\begin{notation}
Denote the page coming from the Lefschetz fibration construction by $\Sigma_0=W\cap \{z_n=\epsilon\}$ and the slightly retracted page coming from the Boothby-Wang construction by $\Sigma_1 = (\partial W-\{z_{n}=t\epsilon | t \in [0,1]\})  \cap \{arg(z_{n})=\theta_{0}\}  \subset S^{2n+1}$.
\end{notation}

\begin{figure}[h]
\begin{overpic}[scale=1,tics=10]{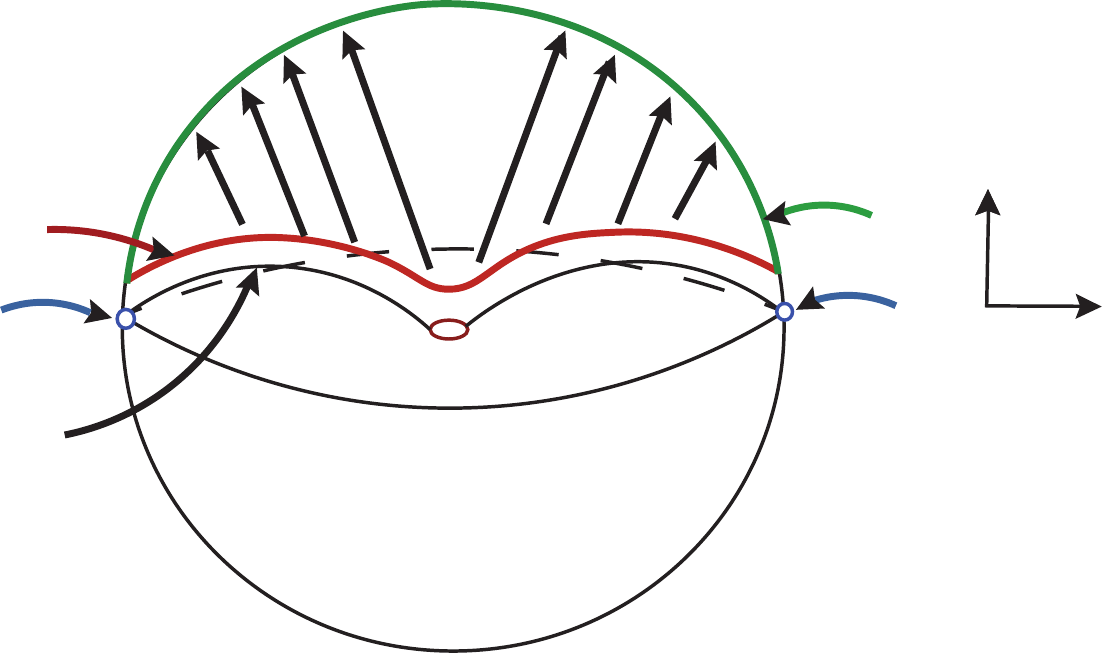}
    \put(82,30){$B$}
    \put(-4,30){$B$}
    \put(40,25){$0$}
     \put(-9,16){$W \cap \{z_n=0\}$}
    \put(35,52){Liouville}
    \put(38, 48){flow}
    \put(-1, 38){$\Sigma_0$}
    \put(80,39){$\Sigma_1$}
    \put(90, 28){$z_0,\dots,z_{n-1}$}
    \put(88, 43){$z_n$}
    \put(55, 27){$W^{2n}$}
    \put(65, 6){$S^{2n+1}$}
    \vspace{7mm}
\end{overpic}
	\caption{Pictorial description of the symplectic isotopy between the pages $\Sigma_0$ and $\Sigma_1$. We label the blue dot as the binding $B$, the red curve as the page $\Sigma_0$, and the green curve as the page $\Sigma_1$. Here 0 denotes the isolated singularity of $f$.}
    \label{proof}
\end{figure}
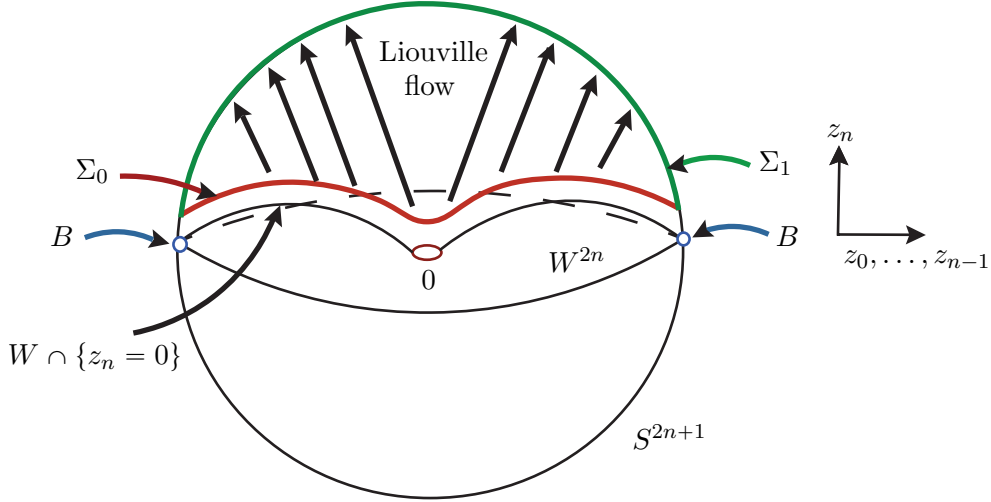

\begin{claim} There exists a symplectic isotopy from $\Sigma_{0}$ to $\Sigma_{1}$ relative to the endpoints, i.e., there exists a one parameter family of $(2n)$-dimensional symplectic submanifolds $(\Sigma_{t}, d\beta|_{\Sigma_t})$, $t \in [0,1]$ as depicted in the Figure \ref{proof}.
\end{claim}\sm

\begin{claimproof} Let $\beta$ be the usual Liouville 1-form on $\mathbb{C}^{n+1}$. Let $X$ and $Y$ be the radial vector field (the Liouville vector field of $\beta$ on all of $\mathbb{C}^{n+1})$ and the rotational vector field in the $z_n=x_n +i y_n$ coordinate, respectively, given as follows: 
\vspace{-.3cm}

\begin{center}
\begin{align*}
X&= \frac{1}{2} \displaystyle\sum_{j=0}^{n} (x_j \frac{\partial}{\partial x_j} + y_j \frac{\partial}{\partial y_j})\\
Y&=\frac{\partial}{\partial \theta}
\end{align*}
\end{center}
\vspace{.3cm}

Consider the flow $gX$, where $g$ is some suitable positive function on Int$(\Sigma_0)$ and is the identity along $\partial \Sigma_0$. Then one can obtain one parameter family of $(2n)$-dimensional submanifolds $\Sigma_t$ by flowing from $\Sigma_0$ to $\Sigma_1$ along the radial vector field $gX$. Furthermore, $\Sigma_{t}$ is transverse to the vector fields $X$ and $Y$ since it is transverse to a multiple of $X$.  Because the boundary of each $\Sigma_t$ is the same, the transversality condition still holds for each hypersurface $\Sigma_t$ at the boundary.\sm

Let $V=\{v_1,\dots,v_{2n}\}$ be a basis of $T\Sigma_t$ at some point $p$. We want to show that $d\beta= \sum_{j=0}^{n}dx_j \wedge dy_j >0$ is symplectic on each $\Sigma_t$ to ensure $\Sigma_t$ is symplectic for $t\in [0,1]$.\sm

We compute $(d\beta)^n$ as follows: Recall that we already know that $(d\beta)^{n+1}(X, Y, V)>0$ since $d\beta$ is the standard symplectic form $\mathbb{C}^{n+1}$. Then 
\begin{equation} 
\begin{aligned}
(d\beta)^{n+1}(X, Y, V) &= n(dx_n \wedge dy_n)(dx_0\wedge dy_0 + \dots + dx_{n-1}\wedge dy_{n-1})^{n} (X, Y, V) \nonumber \\
 &= c (dx_0\wedge dy_0 + \dots + dx_{n-1}\wedge dy_{n-1})^{n} (V)) >0 \hspace{1.1cm} (\star)
\end{aligned}
\end{equation}
implies that $(d\beta)^{n}(V)=(dx_0\wedge dy_0 + \dots + dx_{n-1}\wedge dy_{n-1})^{n}(V)$ is positive since $c>0$.\sm

The last inequality $(\star)$ above is equivalent to the following condition: the projection of $\Sigma_t$ to the first $n$ coordinates (i.e., projection to the $z_0,\dots,z_{n-1}$-plane $Z$) is symplectic.  In other words, $pr:\Sigma_t \to pr(\Sigma_t)$ is a covering space over a symplectic submanifold of the plane $Z$. Also, notice that $\Sigma_0$ and $\Sigma_1$ satisfy $(\star)$ above since $\beta|_{\Sigma_0}$ and $\beta|_{\Sigma_1}$ are Liouville. So each $\Sigma_t$ is symplectic. Therefore, we have a symplectic isotopy relative to the boundary.
 \end{claimproof} \sm

Thus, a right-handed fibered Dehn twist on $\Sigma_1$ is symplectically isotopic to a product of right-handed Dehn twists on $\Sigma_0$. This finishes the proof of Theorem \ref{main_theorem}.

\subsection{Proof of Corollary \ref{main_corollary}}

With Theorem \ref{main_theorem} in hand, we are ready to complete the proof of Corollary \ref{main_corollary}. Here $W= \{f(z)=z_0^k+\dots,z_n^k= \delta \}\cap B^{2n+2}_{\epsilon}$.
   
   \begin{proof}
   Follows from the proof of Theorem \ref{main_theorem}. In this case, the monodromy coming from the Lefschetz fibration is the $k$-th power of a product of right-handed Dehn twists. This follows from the fact that each fiber (page) $F_{\theta}$ has the following behavior:   
\begin{equation} 
\begin{aligned}
F_{\theta} =W \cap \{z_{n}=\epsilon e^{i\theta}\} =\{z_{0}^{k}+\dots+z_{n-1}^{k}=\delta-\epsilon^{k} e^{i\theta k}\}. \nonumber
\end{aligned}
\end{equation}

In other words, the monodromy on $F_{\theta}$ repeats itself $k$ times since it goes over the same $\Phi$ (a product of right-handed Dehn twists) $k$ times as $\theta$ varies. Observe also that the $\theta=0$ case gives us the Lefschetz fiber together with the monodromy $\Phi$. Furthermore, when it sweeps out all of $\theta \in [0,2\pi)$, one gets a right-handed fibered Dehn twist $\Phi_{\partial}$ along the boundary as monodromy which is equal to $\Phi^{k}$ up to symplectic isotopy. \sm

One can Morsify $f$ by adding the linear function $\eta(z_{0}+\dots+z_{n})$ and  observe that degenerate isolated singularity at the origin splits into $(k-1)^n$ many nondegenerate critical points in the fiber. To see this: Take partial derivatives of the Morsified $f(z_0, \dots, z_n)= z_{0}^{k}+\dots+z_{n-1}^{k}+\eta(z_{0}+\dots+z_{n})$ and observe that each partial derivative will generate $(k-1)$ roots. Since there are $n$ variable,  $f$ will have $(k-1)^n$ critical points. Therefore, there are $(k-1)^n$ many nondegenerate critical points, i.e., $\Phi$ is a product of $(k-1)^n$ right-handed Dehn twists. Hence, a right-handed fibered Dehn twist is isotopic to a product of $k(k-1)^n$ right-handed Dehn twists since $\Phi_{\partial}=\Phi^{k}$.
   \end{proof} 
   
   Note that, when $k=1$, $W=\mathbb{C}^{n}$ and a fibered Dehn twist is symplectically isotopic to the identity.  \sm
   
   In the case when $k=2$, $W$ is the cotangent bundle to the $n$-sphere and $\Phi_{\partial}$ is the same as the square of a Dehn twist along the zero section.
   
\subsection{Proof of the Theorem \ref{weightedthm}} Proof follows along the same lines of the proof of Theorem \ref{main_theorem}. Here, unlike Theorem \ref{main_theorem}, $f$ is a weighted homogeneous polynomial with an isolated singularity at 0 and we will study the associated Milnor fibration of the standard sphere $S^{2n-1}$. \smallskip

In this subsection, we wish to prove that the Milnor open book for the standard sphere $S^{2n-1}$ is the contact open book whose monodromy $\Phi_{\partial}$ is a generalization of a fractional fibered Dehn twist along the boundary that is symplectically isotopic to a product of $\mu(f, 0)$ Dehn twists along Lagrangian spheres. \smallskip

Let $z=(z_1, \dots, z_{n})$. Consider the link of $g(z, z_{n+1}):=f(z)-z_{n+1}$ at the origin given as
$$L_{\epsilon}(g)=\{(z,z_{n+1}) \in \mathbb{C}^{n} \times \mathbb{C} \mid g(z,z_{n+1})=0\} \cap S_{\epsilon}^{2n+1}$$ and also its canonical filling given as $$V_{\epsilon}(g)=\{(z,z_{n+1}) \in \mathbb{C}^{n} \times \mathbb{C} \mid g(z,z_{n+1})=0\} \cap B_{\epsilon}^{2n+2}.$$

One can always smooth this filling by using a cutoff function that only depends on $|z|$. Denote the smoothing of $V_\epsilon (g)$ by
$$\tilde{V}_\epsilon (g)=\{g(z,z_{n+1})=\rho \} \cap B_{\epsilon}^{2n+2}.$$

Note that $L_{\epsilon}(g)$ is contactomorphic to the standard contact sphere  $S^{2n-1}$. To see this, let $f(z)=z_{n+1}$ in $L_{\epsilon}(g)$ and observe that this is just a graph of $S^{2n-1}$. Note that one can use the fibration $L_{\epsilon}(g)-\{z_{n+1}=0\}$ sending $(z, z_{n+1}) \mapsto \frac{z_{n+1}}{|z_{n+1}|}$ to study the open book for the $L_{\epsilon}(g)$. Therefore, finding an open book for the standard sphere $S^{2n-1}$ is the same as finding an open book for the boundary of  $\tilde{V}_\epsilon (g)=\{g(z_1,\dots, z_{n+1})=\rho\} \cap B_{\epsilon}^{2n+2}$. Hence, we will use the same techniques as in the proof of the main theorem to observe that the standard contact sphere has the following open book: \sm

\begin{itemize}
\item Page: $\tilde{F}=\{f(z)=\delta\} \cap  B_\epsilon^{2n}$.
\item Binding: $L_\epsilon(f)= \{f(z)=0\} \cap S_\epsilon^{2n-1}$.
\item Monodromy: $\Phi=\Phi_{1}\circ \dots \circ \Phi_{\mu(f,0)}$, where each $\Phi_{i}$ is a right-handed Dehn twist and $\mu(f,0)$ is the Milnor number of $f$.\sm
\end{itemize}

On the other hand, one also needs to describe what the monodromy looks like on the boundary, i.e., prove the following claim. Finally, to finish the proof, we will apply the Liouville flow technique as in the proof of Theorem \ref{main_theorem} to show the symplectic isotopy between two monodromies.\sm

\begin{claim}
The monodromy $\Phi_\partial$ is a generalization of a fractional fibered Dehn twist.
\end{claim}

\begin{proof} Proof is by analyzing a rotation map in the interior and boundary of the page $\tilde{F}$. We will consider $\mathbb{Z}_{d/w_i}$ action on $\tilde{F}$ for $i=1,\dots, n$. \sm
	
Since $f$ is a weighted homogeneous polynomial, we can construct a locally free action on the binding. As for the pages, we need to choose a cutoff function to redefine the hypersurface obtained from the weighted homogeneous polynomial $f$ so that the slightly deformed (by cutting off) hypersurface is Weinstein. Consider a cutoff function $\rho_R$ with the following properties.\sm
	\begin{itemize}
	\item $\rho_R(|z|)=1$ if $|z|\leq R$,
	\item $\rho_R(|z|)=0$ if $|z|>2R$,
	\item $|\rho'_R|<\frac{2}{R}$.
	\end{itemize}
	
	Given a cutoff function $\rho_R$ that only depends on $|z|$, one can smooth the filling $F=\{f(z)=0\} \cap B_{\epsilon}^{2n}$ of $L_\epsilon(f)= \{f(z)=0\} \cap S_\epsilon^{2n-1}$ by considering $f(z)=\rho_R (z)$. Note that $\tilde{F}=\{f(z)=\rho_{R}(z)\} \cap B_{\epsilon}^{2n}$ could be made symplectic by choosing the cut-off function carefully \cite[Section 6.1.2.3]{vK2}. \sm
	
	Consider the following periodic unitary transformation 
\begin{align*}
	\tilde{\Phi}: \tilde{F} &\longrightarrow \tilde{F},\\
	  (z_1, \dots, z_n) &\longmapsto (e^{2\pi i\frac{w_1}{d}}z_1, \dots, e^{2\pi i\frac{w_n}{d}}z_n).
\end{align*}	

Since $f$ is a weighted homogeneous polynomial, the periodic map $\tilde{\Phi}$ is well-defined. Notice also that it is a symplectomorphism since it is a biholomorphism.  \sm

The map $\tilde{\Phi}$ is not the identity map for large $z$. Instead, it can be considered as a generalization of a fractional fibered Dehn twist as defined in \cite{CDK2}. We will elaborate this idea further in this section. It is not (isotopic to) the identity near the boundary of the page $\tilde{F}$. So we need to undo the twist at the boundary to get the identity map and show that it is the desired monodromy. \sm

Firstly, consider the interior of the page $\tilde{F}$, call $F_\theta=\pi^{-1}(e^{i\theta})$. Then one can choose the monodromy on $F_{\theta}$ as $\tilde{\Phi}$ \cite[Lemma 9.4]{M}. In brief, Lemma 9.4 shows that $F_\theta$ given by a weighted homogeneous polynomial $f$ is diffeomorphic to a nonsingular hypersurface and one can construct the above monodromy $\tilde{\Phi}$ on the diffeomorphic copy of $F_\theta$. This completes the monodromy construction in the interior of the page $\tilde{F}$. Note that we need to ensure that this monodromy extends to the whole fibration. In what follows, we discuss this extension. \sm

We wish to observe that the monodromy in the interior of the pages extends to the boundary properly. To see this, we will use the untwisting argument given in Otto van Koert's Ph.D. thesis \cite{vK2} to determine a symplectomorphism of the whole fibration so that the monodromy looks like the identity map near the boundary. For this, let us consider the map $\Phi_t$ induced by the time $t$ flow of the Hamiltonian vector field associated to a Hamiltonian function $H$ on $\tilde{F}$ given by

\begin{equation*}
H(z_1, \dots, z_n)=\sum_j \dfrac{w_j}{2d}|z_j|^2.
\end{equation*}

The time $t$ flow of this Hamiltonian function generates a circle action given by
\begin{equation*}
Fl^{X_H}: (z_1, \dots, z_n) \mapsto (e^{ it\frac{w_1}{d}}z_1, \dots, e^{it\frac{w_n}{d}}z_n).
\end{equation*}

Let $h$ be a function given by $h(z):=(1-\rho_{2R}(|z|))$. Denote by $\tilde{\Phi}_t$, the time $t$-flow of the Hamiltonian vector field associated to $\tilde{H}=h \cdot H$. Then the time $2\pi$-flow of $\tilde{H}|_{\tilde{F}}$, a Hamiltonian on $\tilde{F}$, is the identity for $|z|<2R$, and it induces the following map:
\begin{equation*}
 (z_1, \dots, z_n) \longmapsto (e^{2\pi i\frac{w_1}{d}}z_1, \dots, e^{2\pi i\frac{w_n}{d}}z_n).
\end{equation*}
for $|z|>4R$.\sm

Thus, we can define 
\begin{align*}
\gamma: \tilde{F} &\longmapsto \tilde{F},\\
z &\longmapsto Fl^{X_{\tilde{H}}}_{-2\pi}(z) = (e^{-2\pi i\frac{w_1}{d}}z_1, \dots, e^{-2\pi i\frac{w_n}{d}}z_n).
\end{align*}

Now consider $\Phi_{\partial}= \gamma \circ \tilde{\Phi}: \tilde{F} \longrightarrow \tilde{F}$. Observe that $\Phi_{\partial}$ is not quite a fractional fibered Dehn twist due to the action given by the periodic unitary transformation $\tilde{\Phi}$. We will call it \emph{a generalization of a fractional fibered Dehn twist}. Therefore, the monodromy $\Phi_{\partial}$ is a generalization of a fractional fibered Dehn twist which is the identity at the boundary of $\tilde{F}$ and its neighborhood.
\end{proof}
Note that the  monodromy $\Phi_{\partial}$ is a fractional fibered Dehn twist when $f$ is a homogeneous polynomial as defined in \cite{CDK2}.
   
   \subsection{Proof of the Exact Number of Dehn Twists} \label{number} With Corollary \ref{main_corollary} in hand and \cite{OVK}, we show that there exists $k(k-1)^n$ right-handed Dehn twists for the fibered Dehn twist when $f$ is any homogeneous polynomial of degree $k$ with an isolated singularity at the origin.
   
   \begin{proof} Let $g$ be a homogeneous polynomial of degree $k$ with an isolated singularity at the origin given by $g(z_{0}, \dots, z_{n})=z_{0}^{k}+\dots+z_{n}^{k}$ and let $f$ be a homogeneous polynomial of degree $k$ with an isolated singularity at the origin. We want to show that there exists a 1-parameter family of homogeneous polynomials of degree $k$ with an isolated singularity at the origin, denoted by $p_{t}$ connecting $p_0$ and $p_1$ with the following properties:
\begin{itemize}
\item $p_{0}=f$.
\item $p_{1}=g$.
\item $p_{t}$ is a homogeneous polynomial of degree $k$ with an isolated singularity at the origin.
\item The level set $p_t^{-1}(0) \cap S^{2n+1}$ is smooth for all $t\in I=[0,1]$.
\end{itemize} 
\vspace{1mm}

Firstly, if we can show that the complex codimension of the singular projective variety of homogeneous polynomials of degree $k$ in the projective variety of the space of all homogeneous polynomials of degree $k$ is $\geq$ 1, then there is a path between any two polynomials ($f$ and $g$ in our case) in the nonsingular projective variety. Because one can always travel around the singular locus.\sm 

To see this, let $V_k$ denote the vector space of homogeneous polynomials of degree $k$ and $S_k$ denote the subset of $V_k$ consisting of those homogeneous polynomials of degree k containing a singular locus. Furthermore, we know that homogeneous polynomials of degree $k$ with an isolated singularity at the origin define smooth hypersurfaces in projective space and homogeneous polynomials of degree k whose zero set forms a singular variety in projective space form themselves a singular variety in the space of all homogeneous polynomials of degree $k$.\sm

Note that $S_k$, the zero set of some polynomials on the vector space $V_k$, is a subvariety of $V_k$ and it has complex codimension at least 1 (real codimension 2). Thus, its complement is path connected. Once we get a path, one can locally travel around the singular polynomials. This proves the existence of a 1-parameter family of homogeneous polynomials of degree $k$ with an isolated singularity at the origin as desired.\sm

Now consider $L(p_t):= \set{z\in \mathbb{C}^{n+1} \mid |z|^{2}=1, \, p_t(z)=0 }$, the link of singularity of the polynomial $p_t$ and check that the links $L(p_{t})$ are all contactomorphic. If we can show that they are all contactomorphic, then there must be open books with the same properties. In other words, we have the same number of ($k(k-1)^n$ many) right-handed Dehn twists for the fibered Dehn twist in both Theorem \ref{main_theorem} and Corollary \ref{main_corollary}. \sm

\begin{claim} Let $C:=\set{(t,z)\in I \times \mathbb{C}^{n+1} \mid |z|^2 = 1, \, z\in p_t^{-1}(0)}$. Then the set $C$ is a topologically trivial cobordism between $L(p_0)$ and $L(p_1)$. In particular, $L(p_0)$ and $L(p_1)$ are diffeomorphic.\sm
\end{claim}

\begin{claimproof} Observe that the boundary of $C$ consists of the links $L(p_0)$ and $L(p_1)$. First we show that $C$ is a smooth manifold and hence a smooth cobordism between $L(p_0)$ and $L(p_1)$. To see this, define the map
\begin{align*}
P:I\times\mathbb{C}^{n+1} &\longrightarrow \mathbb{R} \times\mathbb{C},\\
(t, z) &\longmapsto (|z|^2-1, p_t(z)).
\end{align*}

Observe that $C=P^{-1}(0)$. To show the cobordism, we will compute the Jacobian of $P$ and conclude that $C$ is smooth. The Jacobian can be computed as follows:
\begin{equation*}
\begin{pmatrix}
\bar{z} & z \\
\frac{\partial p_t}{\partial z} & 0 \\
0 & \frac{\partial \bar{p_t}}{\partial \bar{z} }
\end{pmatrix}  
\end{equation*} 
since it is generated by $|z|^2-1, p_t, \bar{p}_t$. Then the Jacobian has rank 3 since both $\frac{\partial p_t}{\partial z}$ and $\frac{\partial \bar{p_t}}{\partial z}$ are non-vanishing away from $0 \in \mathbb{C}^{n+1}$ and also 0 is not an element of $P^{-1}(0)$ because of the $|z|^2=1$ condition.\sm

Since the Jacobian of $P$ has real rank 3 at all points in $P^{-1}(0)$, we conclude that $C$ is a smooth submanifold of $I\times\mathbb{C}^{n+1}$ of codimension 3. Hence, $C$ is a smooth cobordism between the links $L(p_0)$ and $L(p_1)$.\sm

To show that this cobordism is topologically trivial, consider the function
\begin{align*}
h: C &\longrightarrow \mathbb{R},\\
(t, z) &\longmapsto t.
\end{align*}

If we can show that $h$ has no critical points in $h^{-1}([0,1])$, we then know that $C$ is a topologically trivial cobordism. To see this, take any point $(t_0, z_0) \in C$ and take a smooth curve $c(t) = (t, z(t))$ in $C$ through $(t_0, z_0)$ with $z(t)$ smooth. One can construct such a curve as follows: Consider a curve in $\mathbb{C}^{n+1}$ given by $$y(t)=z_0+s(t).$$ 

We wish to set up an ODE for $s$ in such a way that $$p_t(z_0+s(t))=0.$$ Differentiating $p_t(z_0+s(t))=0$ will lead us to an ODE. Locally, we can solve the resulting ODE to obtain $y(t)$ satisfying $p_t(y(t))=0$. Now if we let $z(t)=\frac{y(t)}{|y(t)|}$, we get the desired smooth curve $c(t)=(t, z(t)).$\sm

Then $h \circ c$ has non-vanishing derivative at $t=t_0$ since $h\circ c(t)=t$ and $dh$ does not vanish in $(t_0, z_0)$. So no point $(t_0, z_0) \in C$ is a critical point of $h$, i.e. $h$ has no critical points. Thus $C$ is a topologically trivial cobordism. Furthermore, $L(p_0)$ and $L(p_1)$ are diffeomorphic by Theorem 3.1 of \cite{M2}. This concludes the proof of the diffeomorphism between the links $L(p_0)$ and $L(p_1)$.\sm

\end{claimproof}

Since each link $L(p_t)$ is diffeomorphic to $L(p_0)$ and has an associated contact structure $\xi_t$ induced as the link of singularity, Gray's stability theorem shows that there is a contactomorphism sending $\xi_t$ to $\xi_0$. As the resulting contact manifolds are contactomorphic, there must be open books with the same properties. \sm

As there are $k(k-1)^n$ right-handed Dehn twist for the full fibered Dehn twist in the case of Corollary \ref{main_corollary}, we have an open book with the same number of right-handed Dehn twists for the full fibered Dehn twist in the case of Theorem \ref{main_theorem}.
\end{proof}

 \section{The case $n=2$}\label{n_2}
 
In this final section, we combine Corollary \ref{main_corollary} with the monodromy computations appearing in the literature to describe mapping class group relations for some surfaces.

\subsection{The case of arbitrary $k\ge 2$}

\begin{lemma}\cite[Proposition 3.2]{HKP}
For each integer $k \ge 2$, consider the Weinstein domain
\begin{equation*}
W = \{ z_{0}^{k} + z_{1}^{k} + z_{2}^{k} = \delta \} \cap B_\epsilon^{6}
\end{equation*}
Then the page of the open book of $\partial W$ (i.e., the Milnor fiber $F_{k,k}$ associated to the polynomial $z_{0}^{k} + z_{1}^{k}=0$)  is a surface with genus $g = \frac{1}{2}(k-1)(k-2)$ and $k$ boundary components.  It is the `canonical' Seifert surface for the $(k-1,k-1)$ torus link which can be drawn in $\mathbb{R}^{3}$ as shown in Figure 7 \cite{HKP}.  The monodromy for the associated open book decomposition is isotopic to the product
\begin{equation*}
\Phi^{k} = ((D_{1,k-1}\circ\dots D_{k-1,k-1})\circ\dots\circ(D_{1,1}\circ\dots\circ D_{k-1,1}))^k
\end{equation*}
of right-handed Dehn twists $D_{i,j}$ about the curves $\alpha_{i,j} \subset F_{k,k}$ shown in the Figure \ref{n_1_surface}.
\end{lemma}

\begin{figure}[h]
\begin{overpic}[scale=.7]{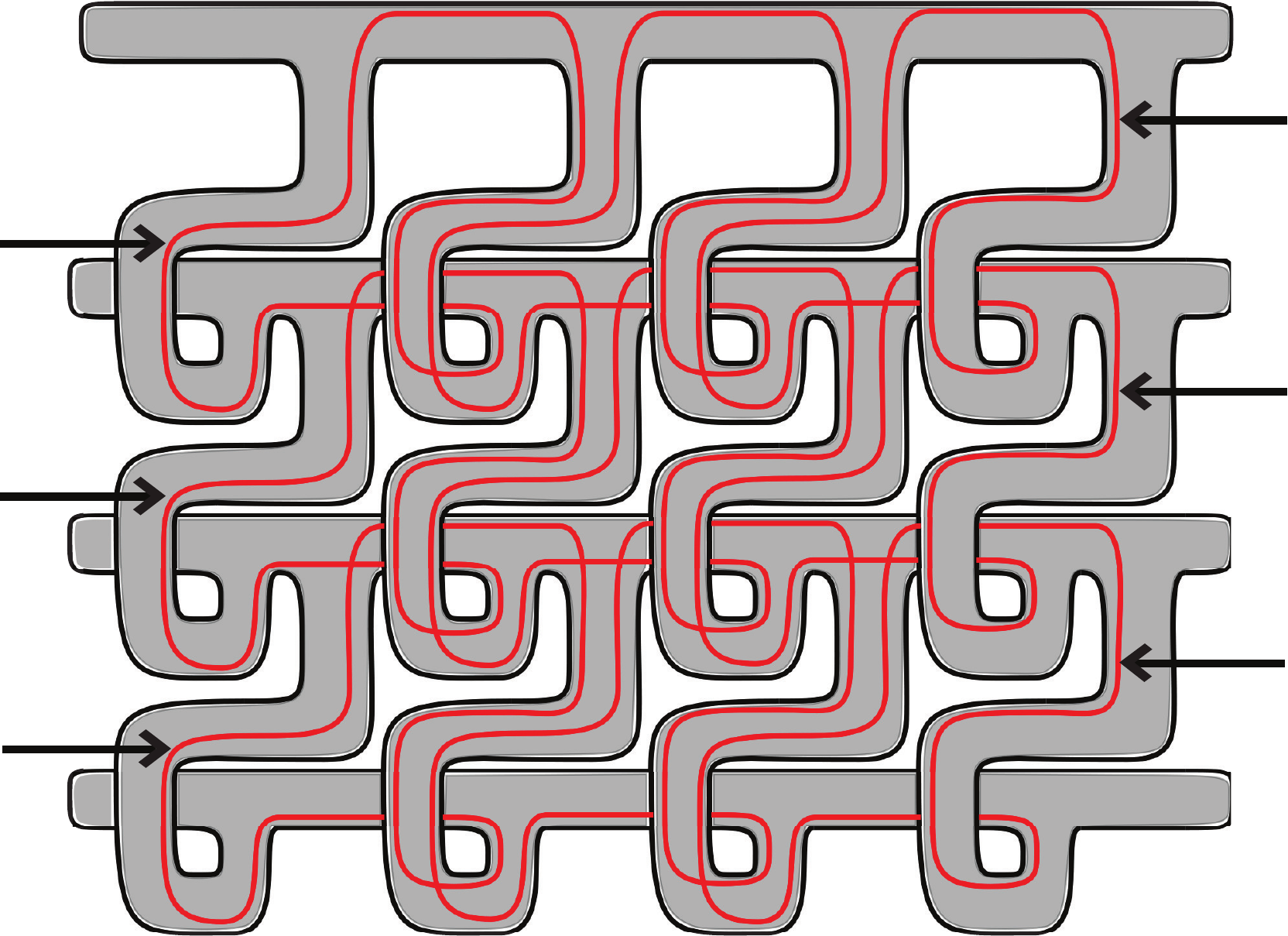}
    \put(-6,53){$\alpha_{1,1}$}
    \put(-6,33.5){$\alpha_{2,1}$}
    \put(-4,29){.}
    \put(-4,26.5){.}
    \put(-4,24){.}
    \put(-4,21.5){.}
    \put(-4,19){.}
    \put(-8.5,14){$\alpha_{k-1,1}$}
    \put(101,62.5){$\alpha_{1,k-1}$}
    \put(101,41.5){$\alpha_{2,k-1}$}
    \put(103,37){.}
    \put(103,34.5){.}
    \put(103,32){.}
    \put(103,29.5){.}
    \put(103,27){.}
    \put(101,20.5){$\alpha_{k-1,k-1}$}
    \vspace{5mm}
\end{overpic}
	\caption{The curves $\alpha_{i,j}$ embedded in the surface $F_{k,k}$.  This figure is essentially the same as Figure 7 in \cite{HKP}.}
    \label{n_1_surface}
\end{figure}
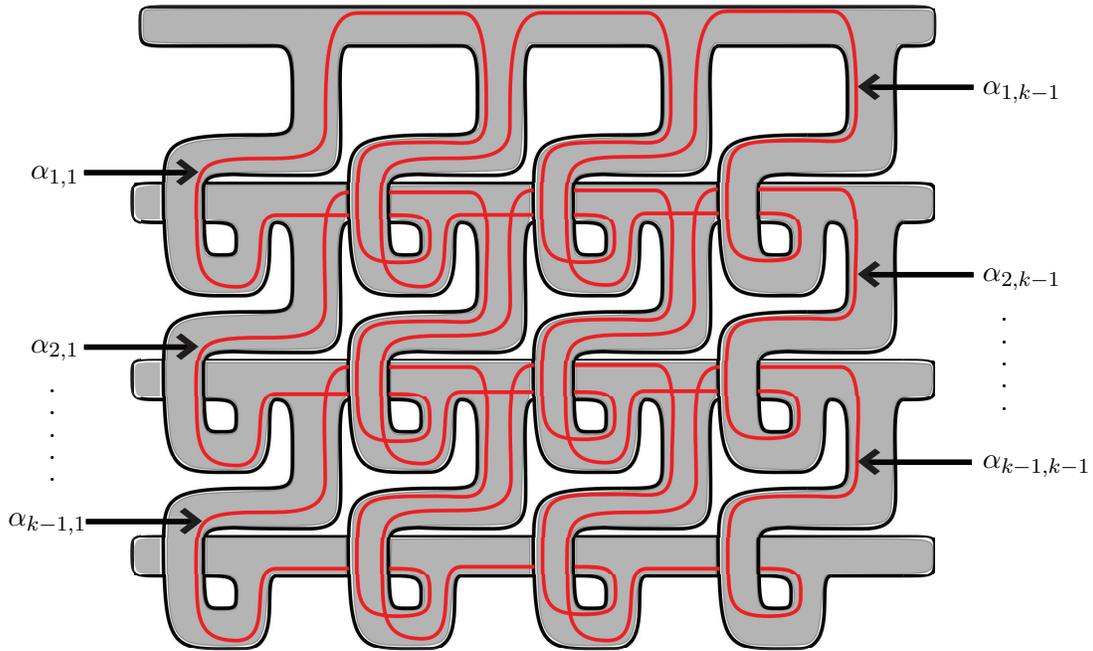

\begin{proof}
The computation of the genus and number of boundary components of $F_{k,k}$ follows immediately from the degree-genus formula for complex curves in $\mathbb{CP}^{2}$.  The drawing of the surface in $\mathbb{R}^{3}$ follows from Example 6.3.10 in \cite{GS}.  The Dehn twist presentation of the monodromy follows from the proof of Proposition 3.2 in \cite{HKP}.\sm
\end{proof}

Combining the above lemma with Corollary \ref{main_corollary}, provides the following:

\begin{corollary}
The surface $F_{k,k}$ described in Figure 7 \cite{HKP} is such that
\begin{equation}
\big{(} (D_{1,k-1}\circ\dots D_{k-1,k-1})\circ\dots\circ(D_{1,1}\circ\dots\circ D_{k-1,1})\big{)}^{k}
\end{equation}
is isotopic to a product of right-handed fibered Dehn twists about each of the boundary components of $F_{k,k}$.\sm
\end{corollary}

Note that in dimension 2, fibered Dehn twists are the same as symplectic Dehn twists so that product of fibered Dehn twists is isotopic to a product of right-handed Dehn twists about curves parallel to the boundary components of $F_{k, k}$. \sm

\subsection{The case $k=3$} \label{Corollary3.1}

Now we describe the case $n=2$, $k=3$ in greater detail, proving Corollary \ref{mapping_class_relationn}. For simplicity, we will henceforth refer to $F_{3,3}$ as $S$.\sm

To complete the proof,  we consider arcs $a_{g}, a_{b}, a_{p}$, and $a_{r}$ in $S$ as depicted in the Figure \ref{cutting_n_1_k_3}.  We also label the boundary components $b_{1},b_{2}$, and $b_{3}$ of $S$ as shown. The orientations of $a$-arcs are indicated with arrows.\sm

\begin{figure}[h]
\begin{overpic}[scale=.6,tics=10]{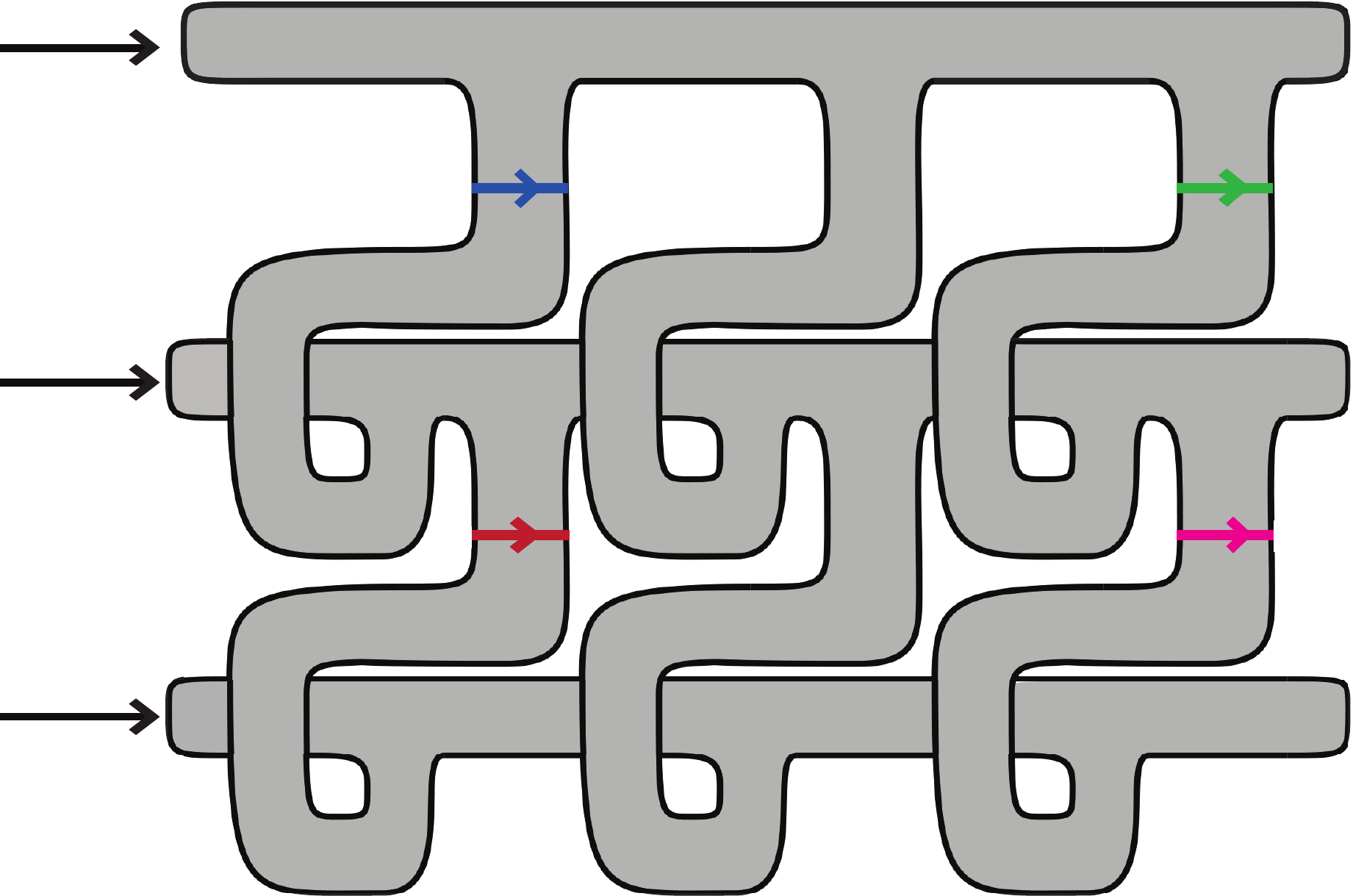}
    \put(-5,62){$b_{1}$}
    \put(-5,37){$b_{2}$}
    \put(-5,12){$b_{3}$}
    \put(37,54.5){$a_{b}$}
    \put(89,54.5){$a_{g}$}
    \put(89,29){$a_{p}$}
    \put(37,29){$a_{r}$}

    \vspace{5mm}
\end{overpic}
	\caption{We label the green, blue, purple, and red arcs as $a_{g}$, $a_{b}$, $a_{p}$, and $a_{r}$, respectively.}
    \label{cutting_n_1_k_3}
\end{figure}
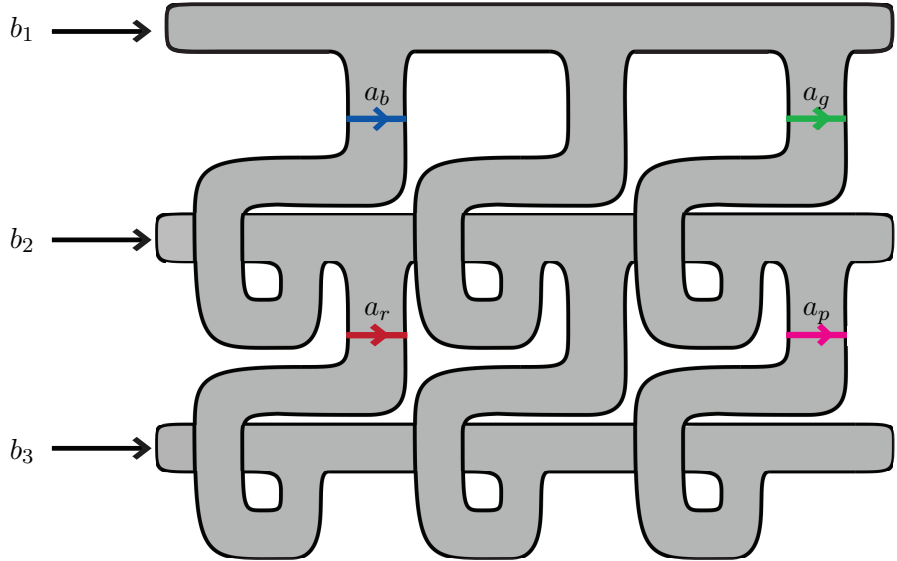

If we cut the surface  $S$ along the arcs $a_{g}, a_{b}, a_{p}$, and $a_{r}$, all that will remain is a 16-gon,  which is topologically just a disk as shown in the Figure \ref{16-gon}. In the Figure \ref{16-gon}, the segments of the boundary are colored according to the identifying colors for the $a$-arcs. The small arrows indicate where the boundary orientation of the 16-gon agrees and disagrees with the orientations of the arcs.\sm

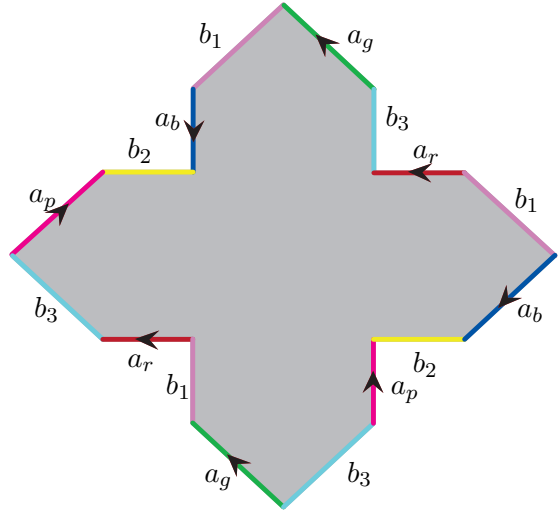
\begin{figure}[h]
\begin{overpic}[scale=1.1,tics=10]{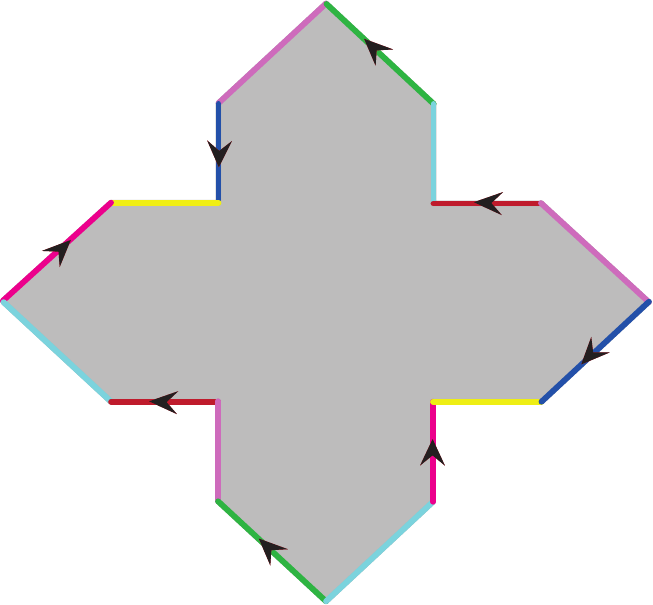}
\vspace{5mm}
\put(23,70){$a_b$}
\put(89,35){$a_b$}
\put(58,85){$a_g$}
\put(32,5){$a_g$}
\put(18,26){$a_r$}
\put(70,64){$a_r$}
\put(0,56){$a_p$}
\put(66,21){$a_p$}
\put(65,70){$b_3$}
\put(88.5,55){$b_1$}
\put(70,25){$b_2$}
\put(58,5){$b_3$}
\put(25,21){$b_1$}
\put(1,35){$b_3$}
\put(18,64){$b_2$}
\put(31,85){$b_1$}

\end{overpic}
    \caption{The disk obtained by taking the closure of $S\setminus(a_{g}\cup a_{b} \cup a_{p} \cup a_{r})$.}
    \vspace{5mm}
    \label{16-gon}
\end{figure}

By regluing along the cut arcs, we obtain an easier-to-visualize depiction of $S$ shown in the Figure \ref{cutting_n_1_k_3}.  Call the diffeomorphism from the surface in Figure \ref{n_1_surface} to the glued up surface in Figure \ref{glued_up_surface}, induced by the cutting and gluing of the $a$-arcs, by $\Psi$. \sm

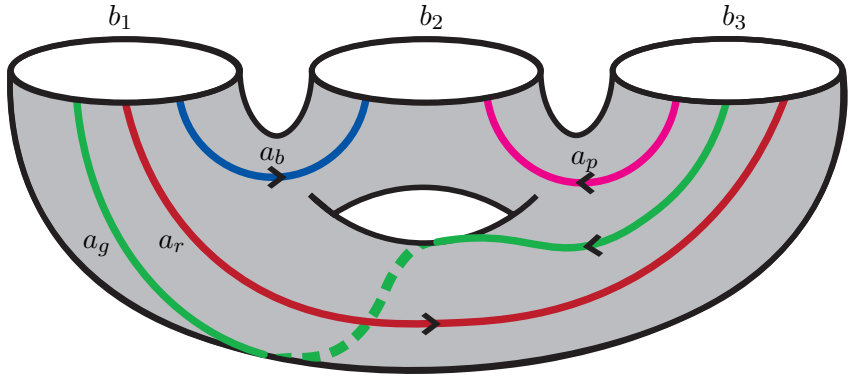
\begin{figure}[h]
\vspace{5mm}
\begin{overpic}[scale=.7,tics=10]{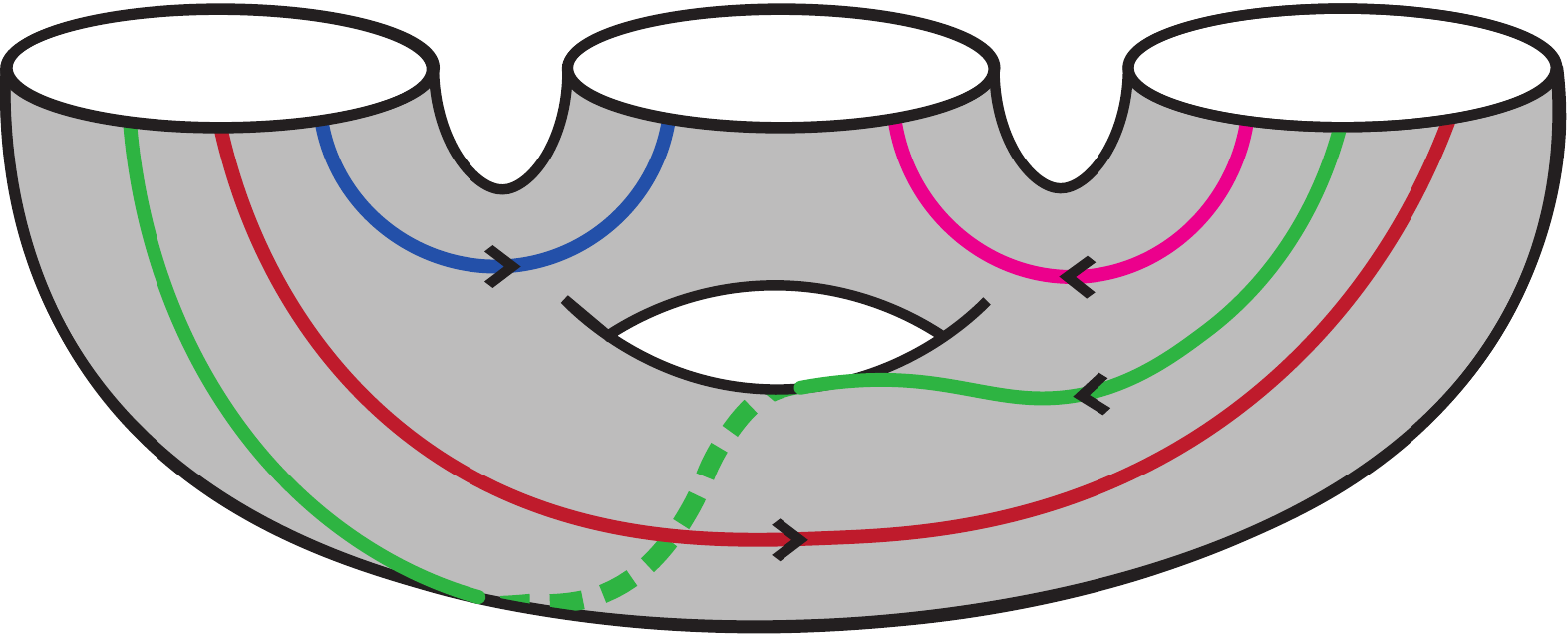}
\put(30,25.5){$a_b$}
\put(67,25){$a_p$}
\put(9,15){$a_g$}
\put(18,15){$a_r$}
\put(12,41.5){$b_1$}
\put(49,41.5){$b_2$}
\put(85,41.5){$b_3$}

\end{overpic}
    \caption{A surface equipped with embedded arcs equivalent to Figure \ref{cutting_n_1_k_3} by the diffeomorphism $\Psi$.}
    \label{glued_up_surface}
\end{figure}

To complete the proof of Corollary \ref{mapping_class_relationn}, one should check that $\Psi$ maps the curves
\begin{equation*}
\alpha_{1,1},\alpha_{1,2},\alpha_{2,1},\alpha_{2,2}
\end{equation*}\sm
of the Figure \ref{n_1_surface} to the curves
\begin{equation*}
\alpha_{b},\alpha_{g},\alpha_{r},\alpha_{p}
\end{equation*}
of the Figure \ref{mapping_class_relation}, respectively. To see this, note that for each color-label $i\in\{g,b,p,r\}$, the simple-closed curve $\alpha_{i}$ intersects $a_{i}$ in a single point, and is disjoint from each $a_{j}$ for $j\ne i$.  These intersection conditions completely determine the isotopy classes of the $\alpha$-curves. This completes the proof of Corollary \ref{mapping_class_relationn}. \sm

\subsection{From Monodromy to Mapping Class Group Relations} \label{MCG} In this section, we study explicit descriptions of the symplectomorphisms in the case $n=2$, which provides a visual representation of the case $n=2$, $k=3$ recovering the classical chain relation for a torus with two boundary components.

\begin{definition} Consider the arcs $\alpha_1, \alpha_2, \alpha_3, \alpha_4, b_1, b_2,$ and $b_3$ in $S_{1, 3}$ as depicted in the Figure \ref{star}. If $\alpha_1, \alpha_2, \alpha_3, \alpha_4, b_1, b_2,$ and $b_3$ are the isotopy classes of simple closed curves in $S_{1, 3}$ then we have the following relation: 

\begin{equation*}
(D_{\alpha_1} \circ D_{\alpha_2} \circ D_{\alpha_3} \circ D_{\alpha_4})^{3} = D_{b_1} \circ D_{b_2} \circ D_{b_3}.
\end{equation*}

\begin{figure}[h]
\begin{overpic}[scale=.7,tics=5]{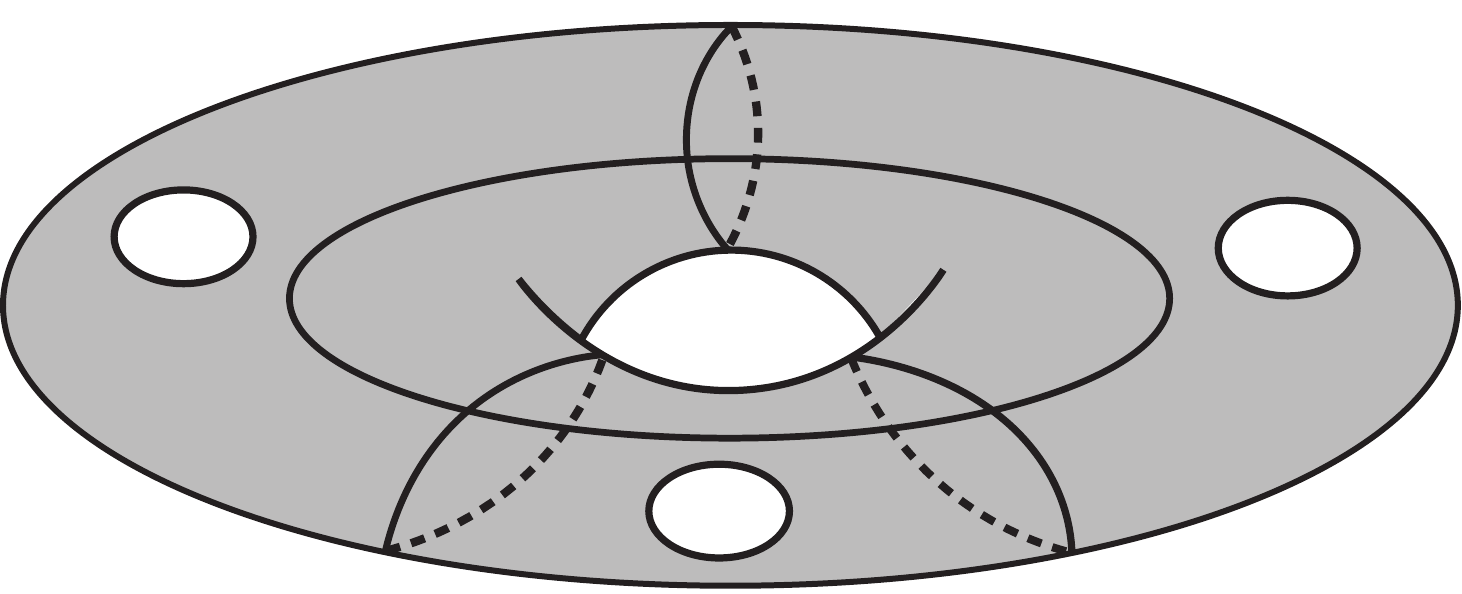}
\vspace{5mm}
\put(90.5,19){$b_1$}
\put(5,19){$b_2$}
\put(52.5,2){$b_3$}
\put(21,7){$\alpha_1$}
\put(71,7){$\alpha_2$}
\put(40,35){$\alpha_3$}
\put(70,28){$\alpha$}
\end{overpic}
    \caption{ We label the boundary components as $b_1, b_2, b_3$ and other components of the isotopy classes of simple closed curves in $S_{1,3}$ as $\alpha_1, \alpha_2, \alpha_3,$ and $\alpha_4$.}
    \vspace{5mm}
    \label{star}
\end{figure}
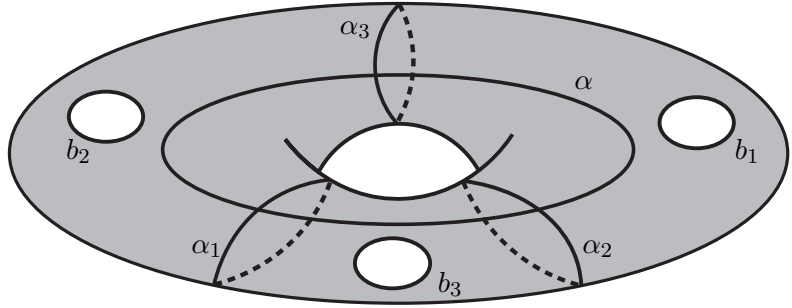
\vspace{.3cm}
\noindent This relation is called the \emph{star relation} \cite{SG}.\sm
\end{definition}

Let $S$ be a compact connected, orientable surface of genus $g \geq 1$ with nonempty boundary and let $\alpha_1, \dots, \alpha_k$ be simple closed curves on $S$. If the intersection number $i(\alpha_j, \alpha_{j+1})=1$ for all $1 \leq j \leq k-1$ and $i(\alpha_j, \alpha_t)=0$ for $|j-t|>1$, then the sequence of isotopy classes of curves $\alpha_1, \dots, \alpha_k$ in $S$ is called a \textit{chain}. \sm

\begin{proposition}\cite[Proposition 4.12]{FM}\label{FMProp} Consider a chain of circles $\alpha_1, \dots, \alpha_k$ in $S$. Denote isotopy classes of boundary curves by $b$ in the even case and by $b_1$ and $b_2$ in the odd case. Then each chain induces one of the following relations in the mapping class group of $S$:
\begin{align*}
(D_{\alpha_1} \circ \dots \circ D_{\alpha_m})^{2m+2} &= D_b \hspace{1cm} \text{if $m$ is even},\\
(D_{\alpha_1} \circ \dots \circ D_{\alpha_m})^{m+1} &= D_{b_1} \circ D_{b_2} \hspace{.3cm} \text{if $m$ is odd}.
\end{align*}
\end{proposition}
In each case, the relation is called a \emph{chain relation} or a \emph{k-chain relation}. There is another version of the chain relation described in \cite{FM} which will be useful in our case:
\begin{align*}
(D_{\alpha_1}^2 \circ D_{\alpha_2} \circ \dots \circ D_{\alpha_m})^{2m} &= D_b \hspace{.9cm} \text{if $m$ is even},\\
 (D_{\alpha_1}^2 \circ D_{\alpha_2} \circ \dots \circ D_{\alpha_m})^{m} &= D_{b_1} \circ D_{b_2} \hspace{.3cm} \text{if $m$ is odd}.
\end{align*}

Notice how the Dehn twist along the boundary differs in each case. Now consider any two arcs, $\alpha$ and $\beta$ on S with geometric intersection number 1 (i.e., the minimum number of intersections of curves $\alpha^{'}$ and $\beta^{'}$ isotopic to $\alpha$ and $\beta$, respectively). Then the following relation is called the \emph{braid relation}:
\begin{equation*}
D_{\alpha} \circ D_{\beta} \circ D_{\alpha} = D_{\beta} \circ D_{\alpha} \circ D_{\beta}.
\end{equation*}\sm
\indent It is known that the star relation implies the 2-chain relation $(D_{\alpha_1}D_{\alpha_2})^6 = D_b$  by using the braid relation. If one can show that the relation
\begin{equation*}
(D_{\alpha_{r}}\circ D_{\alpha_{p}} \circ D_{\alpha_{b}} \circ D_{\alpha_{g}} )^{3} = D_{b_{1}}\circ D_{b_{2}} \circ D_{b_{3}}.
\end{equation*}\sm
in the mapping class group of $S_{1,3}$ given in Corollary \ref{mapping_class_relationn} is isotopic to the star relation, then it satisfies the 2-chain relation. To see this, we will use the following argument explained by Dan Margalit \cite{DM}.\sm

Consider the yellow curve $\alpha_y$  as depicted in the Figure \ref{equivalence} 
which intersects $\alpha_g$ once and is disjoint from $\alpha_b$ and $\alpha_p$. Then we have:
\begin{equation*}
D_{\alpha_r} = D_{\alpha_g}^{-1} \circ D_{\alpha_y} \circ D_{\alpha_g}.
\end{equation*}

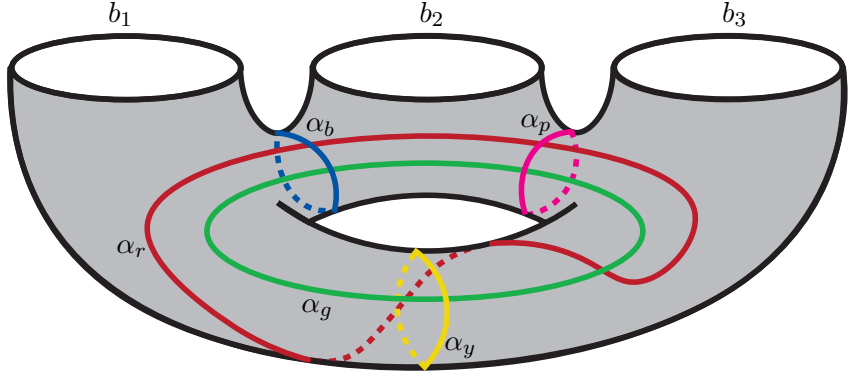
\begin{figure}[h]
\vspace{5mm}
\begin{overpic}[scale=.7,tics=10]{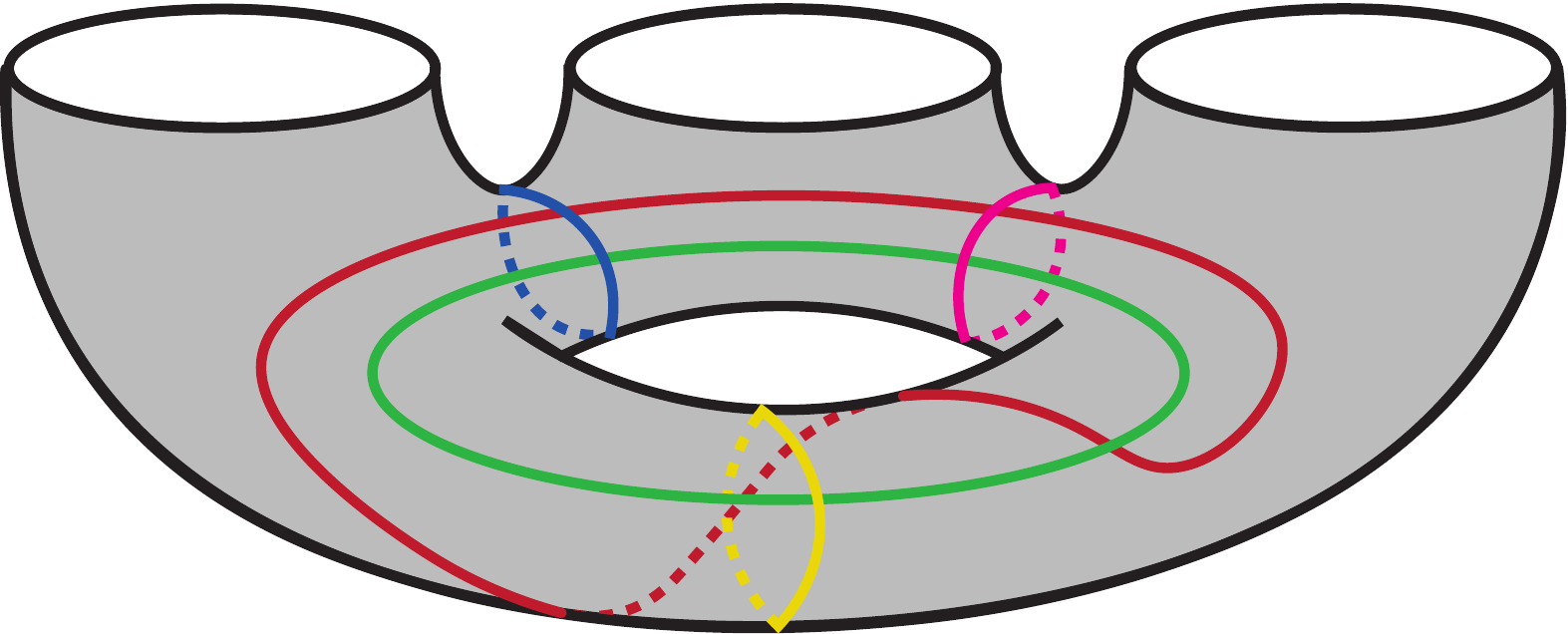}
\put(35.5,28.5){$\alpha_b$}
\put(61,29){$\alpha_p$}
\put(35,7){$\alpha_g$}
\put(13,14){$\alpha_r$}
\put(52,2.5){$\alpha_y$}
\put(12,41.5){$b_1$}
\put(49,41.5){$b_2$}
\put(85,41.5){$b_3$}
\end{overpic}
	\caption{We label the blue, green, purple, red, and yellow curves as $\alpha_{b}$, $\alpha_{g}$, $\alpha_{p}$, $\alpha_{r}$, and $\alpha_y$, respectively.  The boundary components are labeled as $b_{1}$, $b_{2}$, and $b_{3}$.}
    \label{equivalence}
\end{figure}

After plugging $D_{\alpha_r} = D_{\alpha_g}^{-1} \circ D_{\alpha_y} \circ D_{\alpha_g}$ into the relation given in Corollary \ref{mapping_class_relationn}, we get the following relation:
\begin{equation*}
((D_{\alpha_g}^{-1} \circ D_{\alpha_y} \circ D_{\alpha_g}) \circ D_{\alpha_p} \circ D_{\alpha_b} \circ D_{\alpha_g})^3 = D_{b_{1}}\circ D_{b_{2}} \circ D_{b_{3}}.
\end{equation*}

After multiplying out $D_{\alpha_g}^{-1} \circ D_{\alpha_y} \circ D_{\alpha_g} \circ D_{\alpha_p} \circ D_{\alpha_b} \circ D_{\alpha_g}$ with itself three times and canceling out the terms, we obtain:
\begin{equation*}
D_{\alpha_g}^{-1} \circ D_{\alpha_y} \circ D_{\alpha_g} \circ D_{\alpha_p} \circ D_{\alpha_b} \circ D_{\alpha_y} \circ D_{\alpha_g} \circ D_{\alpha_p} \circ D_{\alpha_b} \circ D_{\alpha_y} \circ D_{\alpha_g} \circ D_{\alpha_p} \circ D_{\alpha_b} \circ D_{\alpha_g}= D_{b_{1}}\circ D_{b_{2}} \circ D_{b_{3}}.
\end{equation*}

Using the fact that $D_{b_{1}}\circ D_{b_{2}} \circ D_{b_{3}}$ is central, we can conjugate both sides by $D_{\alpha_{g}}$ and get:
\begin{equation*}
D_{\alpha_y} \circ D_{\alpha_g} \circ D_{\alpha_p} \circ D_{\alpha_b} \circ D_{\alpha_y} \circ D_{\alpha_g} \circ D_{\alpha_p} \circ D_{\alpha_b} \circ D_{\alpha_y} \circ D_{\alpha_g} \circ D_{\alpha_p} \circ D_{\alpha_b} = D_{b_{1}}\circ D_{b_{2}} \circ D_{b_{3}}.\sm
\end{equation*}

We can also apply conjugation by  $D_{\alpha_{y}}$. After applying conjugation, we have:
\begin{equation*}
D_{\alpha_g} \circ D_{\alpha_p} \circ D_{\alpha_b} \circ D_{\alpha_y} \circ D_{\alpha_g} \circ D_{\alpha_p} \circ D_{\alpha_b} \circ D_{\alpha_y} \circ D_{\alpha_g} \circ D_{\alpha_p} \circ D_{\alpha_b} \circ D_{\alpha_y} = D_{b_{1}}\circ D_{b_{2}} \circ D_{b_{3}},
\end{equation*}
\begin{equation*}
(D_{\alpha_g} \circ D_{\alpha_p} \circ D_{\alpha_b} \circ D_{\alpha_y})^3 = D_{b_{1}}\circ D_{b_{2}} \circ D_{b_{3}}
\end{equation*}
which is the star relation. \sm

Hence, $n=2$ and $k=3$ case of Corollary \ref{main_corollary} recovers the star relation. The star relation and the braid relation imply the 2-chain relation as follows. \sm

If you cap off one boundary component, then two curves in the star relation become the same. Then it is an application of the braid relation to get the 2-chain relation $$(D_{\alpha_1}^{2}\circ D_{\alpha_2} \circ D_{\alpha_3})^3 = D_{b_{1}} \circ D_{b_{2}}.$$ 

Therefore, $n=2$ and $k=3$ case of Corollary \ref{main_corollary} recovers the classical chain relation for a torus with 2 boundary components.

\bibliographystyle{amsalpha}
\bibliography{paperbib}
\end{document}